\documentclass[a4paper,11pt,reqno]{amsart}

\pdfoutput=1

\usepackage{a4wide,color,array}
\usepackage{cite}
\usepackage{hyperref}
\usepackage{amsmath,amssymb,amsthm,color}
\hypersetup{linktocpage,colorlinks, citecolor={magenta}}
\usepackage[english]{babel}
\usepackage[utf8]{inputenc}

\usepackage{xspace}
\usepackage{bm}				
\usepackage{bbm,upgreek}		
\usepackage[e]{esvect}		
\usepackage{esint}			
\usepackage{etoolbox}			
\usepackage{calc}				
\usepackage{eso-pic}			
\usepackage{datetime}			

\usepackage{lmodern}
\numberwithin{equation}{section}

\usepackage{enumitem}
\setlist{nosep}
\setlist{noitemsep}

\newcommand{\R}{\mathbb{R}}

\newtheorem{theorem}{Theorem}
\newtheorem{proposition}{Proposition}[section]
\newtheorem{lemma}[proposition]{Lemma}
\newtheorem{corollary}[proposition]{Corollary}
\newtheorem{remark}[proposition]{Remark}

\theoremstyle{plain}

\theoremstyle{definition}
\newtheorem{defi}[proposition]{Definition}

\newcommand{\tref}[1]{Theorem~\ref{t.#1}}
\newcommand{\pref}[1]{Proposition~\ref{p.#1}}
\newcommand{\lref}[1]{Lemma~\ref{l.#1}}
\newcommand{\cref}[1]{Corollary~\ref{c.#1}}

\newcommand{\eref}[1]{(\ref{e.#1})}

\def \supp{\mathrm{supp }} 
\def \1{\mathbf{1}} 
\def \mcl{\mathcal}

\def \ep{\varepsilon}
\def \dist{\mathrm{dist}}
\newcommand{\g}{\mathsf{g}}

\def\({\left(}
\def\){\right)}

\def\XXint#1#2#3{{\setbox0=\hbox{$#1{#2#3}{\int}$}
		\vcenter{\hbox{$#2#3$}}\kern-.5\wd0}}

\newcommand{\fluct}{\mathrm{fluct}}
\newcommand{\Fluct}{\mathrm{Fluct}}

\newcommand{\E}{\mathbb{E}}
\renewcommand{\P}{\mathbb{P}}

\newcommand{\ov}{\overline}

\renewcommand{\tilde}{\widetilde}

\makeatletter
\def\namedlabel#1#2{\begingroup
	#2%
	\def\@currentlabel{#2}%
	\phantomsection\label{#1}\endgroup
}
\makeatother

\def \d {\mathsf{d}} 
\def \s {\mathsf{s}} 
\def \c {\mathsf{c}} 
\def \cd{\mathsf{c}_{\d}} 
\def \cds{\mathsf{c}_{\d, \s}} 
\def \h{\mathsf{h}} 
\def \Zpart{\mathsf{Z}} 
\def \Kpart{\mathsf{K}} 
\def \Fenergy{\mathsf{F}} 
\def \Error {\mathrm{Err}}

\def \emp {\mathsf{emp}}  

\title{Poisson Statistics for Coulomb Gases at Intermediate Temperature Regimes}
\author{David Padilla-Garza}
\address[D. Padilla-Garza]{Einsten Institute of Mathematics, Hebrew University of Jerusalem}
\email{David.Padilla-garza@mail.huji.ac.il}

\author{Luke Peilen}
\address[L.Peilen]{Department of Mathematics, Temple University}
\email{luke.peilen@temple.edu}

\author{Eric Thoma}
\address[E.Thoma]{Department of Mathematics, Stanford University}
\email{thoma@stanford.edu}
\date{\today}

\begin{document}
	\begin{abstract}
		We consider the microscopic statistics of a Coulomb gas in $\R^2$ at intermediate temperatures. In particular, we show that the microscopic point process associated with the Coulomb gas converges to a homogeneous Poisson point process at intermediate temperature regimes $\beta N \rightarrow \infty$ and $\beta \sqrt{N} \log N \rightarrow 0$, extending previous results. Our approach relies on a novel quantitative asymptotic description of correlation functions, which is of its own interest.
	\end{abstract}
	
	\maketitle

	\section{Introduction}
	We are interested in the Coulomb gas in $\R^\d$, an interacting particle system whose equilibrium statistics are governed by the Gibbs measure at inverse temperature $\beta$
	\begin{equation}\label{e.Gibbs1}
		\P_{N,\beta}(dX_N)=\frac{1}{\Zpart_{N,\beta}}e^{-\beta \mathcal{H}_N(X_N)}~dX_N
	\end{equation}
	where $X_N=(x_1,\ldots,x_N)\subset (\R^\d)^N$ denotes the particle configuration and $\mathcal{H}_N$ is an interaction Hamiltonian given by 
	\begin{equation}\label{e.Hamiltonian}
		\mathcal{H}_N(X_N)=\frac{1}{2}\sum_{i \ne j}\g(x_i-x_j)+N\sum_{i=1}^N V(x_i).
	\end{equation}
	$V:\R^\d \rightarrow \R$ is an external confinement potential, assumed to grow rapidly enough at infinity so that the particles remain asymptotically confined to a compact subset of $\R^\d$ with overwhelming probability. The kernel $\g$ we consider is the Coulomb kernel
	\begin{equation}\label{e.Coulomb}
		\g(x)=\begin{cases}
			-\log|x| & \text{if }\d=2, \\
			|x|^{2-\d} & \text{if }\d \geq 3.
		\end{cases}
	\end{equation}
	There is also interest in studying more general Riesz gases, where the interaction power is allowed to vary between $0$ and $\d$:
	\begin{equation}\label{e.Riesz}
		\g(x):=\g_\s(x)=\begin{cases}
			|x|^{-s} & \text{if }0<\s<\d, \\
			-\log|x| & \text{if }\s=0.
		\end{cases}
	\end{equation}
	The reason for taking $\s<\d$ is to preserve local integrability of the kernel at the origin.
	In the Coulomb case $\g$ satisfies $-\Delta \g=\cd \delta_0,$ unlocking access to a trove of elliptic techniques. 
	
	For the Riesz gases, one has a similar relationship with a fractional operator, namely $(-\Delta)^{\frac{\d-\s}{2}}\g=\cds \delta_0$ in the sense of distributions. The \textit{superCoulombic} Riesz gas, for $\d-2<\s<\d$ is then particularly tractable. The fractional Laplacian with power $0<\alpha<1$ has an interpretation as a (possibly degenerate) elliptic operator in $\R^{\d+1}$ via the Caffarelli-Silvestre extension procedure (see \cite{CS07}) 
    When $\s<\d-2$ (the \textit{subCoulombic} Riesz gas) the fractional operator is extremely nonlocal and the analysis of the gas is rather different.

     In this article, we will focus on the Coulomb gas, where the superharmonicity of the kernel $\g$ plays a key role. We will work mainly in dimension $\d=2$. We will also focus on the temperature regime $\frac{1}{N} \ll \beta \ll 1$, which we call the \emph{intermediate} temperature regime. 
	

	\subsection{The Thermal Equilibrium Measure}
	From \cite{F35} (see also \cite{ST97}), if one assumes that the potential $V$ is everywhere finite, lower semicontinuous, and grows faster than the logarithm at infinity, i.e.
	\begin{equation}\label{e.potgrowth}
		\lim_{|x|\rightarrow \infty}V(x)+\g(x)=+\infty
	\end{equation}
	then the mean-field energy
	\begin{equation}\label{e.E}
		\mathcal{E}_V(\mu):=\mcl E(\mu)+\int V(x)~\mu(dx):=\frac{1}{2}\iint \g(x-y)~\mu(dx)\mu(dy)+\int V(x)~\mu(dx)
	\end{equation}
	has a 
    compactly supported minimizer among probability measures known as the \textit{equilibrium measure}, which we denote $\mu_V$.
    The equilibrium measure satisfies the Euler-Lagrange equation
    \begin{equation}
    \label{eq:ELeq}
        \begin{cases}
            \h^{\mu_{V}} + V &\geq c_{V} \quad  \text{ in }\mathbb{R}^{\d} \\
            \h^{\mu_{V}} + V &= c_{V} \quad  \text{ in }\mathrm{supp} \(\mu_{V} \)
        \end{cases}
    \end{equation}
    for some constant $c_V$, where $\h^{\nu}$ is the Coulomb potential
	\begin{equation}\label{e.potential}
		\h^{\nu}(x):=\int \g(x-y)~\nu(dy).
	\end{equation}
    The equilibrium measure is a good approximation of the \textit{empirical measure} in the sense that 
	\begin{equation*}
	 \emp_N:=\frac{1}{N}\sum_{i=1}^N \delta_{x_i}\rightarrow \mu_V
	\end{equation*}
	in a large deviations sense so long as $V$ is lower semicontinuous and $\beta \gg \frac{1}{N}$ (see for instance \cite{GZ19}). At temperatures $\beta \simeq \frac{1}{N}$ the gas undergoes thermalization, and this equilibrium measure is no longer an adequate approximation. Entropic effects appear at leading order, and particle locations are no longer strongly confined to a compact set; instead, one needs to consider a \textit{thermal equilibrium measure} $\mu_\theta$ minimizing
	\begin{equation}\label{e. thermeq}
		\mcl E_\theta(\mu):= \mcl E_V(\mu)+\frac{1}{\theta}\int \mu \log \mu
	\end{equation}
	with $\theta:=\beta N$. This minimizer exists (see \cite{Lam21}) and satisfies the Euler-Lagrange equation
	\begin{equation}\label{e. EL}
		\h^{\mu_\theta}+V+\frac{1}{\theta}\log \mu_\theta=c_\theta
	\end{equation}
	for some constant $c_\theta$. Furthermore, one has a law of large numbers (see \cite{GZ19}) for the empirical measures
	\begin{equation}\label{eq: LLN}
		\emp_N \rightarrow \mu_\theta
	\end{equation}
	in a large deviations sense when $\beta N \to \theta$ as $N \to \infty$. Indeed, the authors in \cite{AS21} argue that $\mu_\theta$ is a more precise approximation of $\mu_V$ at all temperature regimes.  Concentration of measure is studied in \cite{PG23} and quantiative approximation of $\mu_V$ by $\mu_\theta$ at temperatures $\beta \gg \frac{1}{N}$ is established in \cite{AS22} via connections between the equilibrium problem and the classical obstacle problem (see \cite{S24} and \cite{C98}).
		
	\subsection{Main results}
	Our main goal in this paper is to better understand the microscopic behavior of the gas. A useful means of examining this behavior is by way of the \textit{microscopic point process} near a point $\ov z \in \R^\d$, which is defined by  	\begin{equation}\label{e.LPP}
		Q_{\ov z, N} = \sum_{i=1}^N \delta_{N^{1/\d}(x_i - \ov z)}. 
	\end{equation}
	Effectively, $Q_{\overline{z},N}$ centers the configuration around $\ov z$ and then blows up the distance between particles so that the typical interparticle distance (previously $\simeq N^{-1/\d}$) is now order one. We will study the convergence as $N \to \infty$ of $Q_{\overline{z},N}$ in the weak topology on point processes, Definition \ref{d. weak conv}, 
    generated by integration against continuous, compactly supported functions on $\R^\d$.
		
	We will consider $\beta = \beta_N$ as a sequence dependent on $N$ and let $\theta_\ast = \inf_N N \beta_N$. Throughout the paper, we will assume that $\d = 2$, $\theta_\ast > 0$, and
	\begin{equation} \label{e.conf} \tag{Conf} \begin{aligned}
		 &\int_{|x|\geq 1}e^{-\frac{\theta_\ast}{2}(V(x)+\g(x))}~dx+\int_{|x|\geq 1}e^{-\theta_\ast(V(x)+\g(x))}|x|\log^2|x|~dx<+\infty,  \\ 
		 &V \in C^2 \quad \text{and} \quad \Delta V \geq \c>0 \quad \text{in a neighborhood of } \supp(\mu_V), \\ 
		 &\lim_{|x| \to \infty} V(x) - \log |x| = +\infty.
	\end{aligned}\end{equation}

	For convenience, we will sometimes assume stronger growth of $V$ at $\infty$ and an upper bound for the Laplacian of $V$, namely
	\begin{equation} \label{e.growth} \tag{Grow}
            \begin{aligned}
		&V(x) \geq |x|^{\frac{2}{99}} \quad \forall |x| \geq C, \\
            &\sup_{x \in \R^2} \Delta V(x) \leq C.
        \end{aligned}
	\end{equation}
	   Both assumptions may be weakened at the cost of more involved error terms. 
We will also, for our main results but not every intermediate result, assume that the temperature is sufficiently high. Precisely, we assume
	\begin{equation} \label{e.temp.zero} \tag{Temp}
		\beta_N = o\left (\frac{1}{N^{1/2} \log N}\right ) \quad \text{as} \quad N \to \infty.
	\end{equation}

	We allow implicit constants throughout to depend on the parameters within \eref{conf}, \eref{growth}, and \eref{temp.zero}. We will also implicitly assume $N$ is large relative to these parameters.
	
	\begin{theorem}\label{t.1}
		Assume \eref{temp.zero} and \eref{growth}, and let $\ov z\in\mathrm{int}(\supp(\mu_V))$. Then the local point process $Q_{\ov z, N}$ under $\P_{N,\beta}$ converges weakly to a Poisson point process with intensity $\mu_V(\ov z)\mathrm{Leb}$ for the Lebesgue measure $\mathrm{Leb}$ on $\R^{\d}$.
	\end{theorem}
	We remind the reader of the definition of the Poisson point process and weak convergence of point processes in Section \ref{sect:convergence}. Assumption \eref{conf} allows us to access the confinement bound
	\begin{equation}\label{e.strong confinement}
		\P(\exists ~|x_i|>L) \leq CN \int_{|x|>L}e^{-\beta N\zeta(x)}~dx
	\end{equation}
	to obtain sufficient localization, where $\zeta = \h^{\mu_V} + V \geq V - \log |x| - C$, as established in \cite[Theorem 3]{T25}. Our assumptions on $V$ also guarantee that we can make use of the quantitative approximation of $\mu_V$ by $\mu_\theta$, established in \cite[Theorem 1]{AS22}, which in particular implies
	\begin{equation}\label{e. inttemp}
		\lim_{\theta \rightarrow \infty}\mu_\theta (x)=\mu_V(x)
	\end{equation}
	for $x \in \mathrm{int}(\supp(\mu_V))$. 
    
Let us briefly contextualize the setup for Theorem \ref{t.1}. In general, the process (\ref{e.LPP}) is hard to study due to the long-range nature of the Coulomb energy. The authors in \cite{LS15} studied the push-forward of $Q_{\overline{z},N}$ under an averaging operator and identified a critical temperature scale, $\beta \simeq N^{\frac{2}{\d}-1}$ $(\beta \simeq N^{-\frac{\s}{\d}}$ for Riesz interactions with kernels $\g_\s:=|x|^{-\s}$) at which the averaged point process converges as $N \rightarrow \infty$ to a point process that minimizes a free energy with competing energy and entropic terms. For the $\d=1$ log-gas this process is the celebrated $\mathrm{Sine}_\beta$ point process, whose convergence has been well-studied (cf. \cite{VV09}, \cite{KS09}). For the Coulomb gas in $\d=2$, this corresponds to the Ginibre point process at $\beta=2$, but is otherwise not identified. At regimes $\beta \gg N^{\frac{2}{\d}-1}$ one would correspondingly expect the energy term to dominate and see particles concentrate around energy minimizers, and as $\beta \downarrow 0$ faster than $N^{\frac{2}{\d}-1}$ one would expect to see entropic behavior dominate. This corresponds to a kind of complete independence of the particles, and convergence to a Poisson point process (see $\S 3$ for a more thorough discussion). Indeed, this regime was studied in \cite{L16}, where the author established convergence of this averaged point process to a Poisson point process. 
	
	In the high temperature regime $\beta \simeq \frac{1}{N}$ where the strength of the interaction is weaker and particles are no longer confined to a compact region in $\R^\d$, the analysis is a bit more tractable and the corresponding entropic domination has been well-studied. In particular, the authors in \cite{Lam21} have shown convergence of the local point process $Q_{\ov z, N}$ to a Poisson point process for a very general collection of interacting particle systems that include Coulomb and Riesz gases, generalizing work for the $\beta$-ensembles initiated in \cite{BGP15} and \cite{NT20}.
	
	The behavior of the local point process at intermediate temperature regimes $N^{-1}\ll \beta \ll N^{-\frac{\s}{\d}}$ for the Riesz gas is largely unknown, outside of the case of the Gaussian $\beta$-ensembles (the $1-\d$ log gas with quadratic confinement potential) studied in \cite{BGP15}. Theorem \ref{t.1} begins to fill in this gap in the literature for the $\d=2$ Coulomb gas.

    \begin{remark}
    Assumptions (\ref{e.growth}) may be weakened with some added technical difficulty. It is particularly simple, for example, to weaken \eref{growth} to $V(x) \geq |x|^{\gamma}$ for some $\gamma > 0$ for large $|x|$ at the sole cost of changing certain implicit constants $C$. It is also clear that \tref{1} remains true after modifying the potential far from the origin so long as $\beta \geq C^{-1} N \log N$, since with high probability all particles lie in a fixed compact set and one can therefore modify the potential outside this set without significantly changing microscopic statistics.
    \end{remark}
	
	We establish Theorem \ref{t.1} by a quantitative comparison of the correlation functions for $Q_{\overline{z},N}$ to those of a Poisson point process of intensity $\mu_V(\overline{z})$; see Proposition \ref{p. corrfunctions} below. This in particular allows us to estimate various microscopic statistics of $\P_{N,\beta}$ at finite $N$.
	\begin{corollary}
		Assume \eref{temp.zero}. Let $x \in \mathrm{int}(\supp(\mu_V))$ and let $\rho_1(x)dx$ denote the first marginal $\P_{N,\beta}(x_1 \in dx)$. Then for any $\gamma > 0$, we have
		\begin{equation}\label{e.onepoint}
			\rho_1(x)=\mu_\theta(x)\left(1+O\left(\beta N^{\frac{1 + \gamma}{2}}\right)\right)
		\end{equation}
		as $N \to \infty$. 
		In particular, applying \eref{onepoint} and assuming in addition that $\beta \leq N^{-\frac{1}{2}-\alpha}$ for some $\alpha>0$ and $N$ large enough, we have for any open set $U$ containing $\supp(\mu_V)$ that
		\begin{equation}\label{e.confinement result}
			\P_{N,\beta}(\exists ~x_i \notin U) \leq CN\mu_\theta(U^c) \leq Ce^{-c\beta N \min \left(\dist(U^c, \supp(\mu_V))^2, 1 \right)},
		\end{equation}
		for some $c > 0$. 
	\end{corollary}
	We note that \eref{confinement result} is proved for general $\beta$ and in a form effective on microscopic scales in \cite{T25}. Our proof is however essentially independent of \cite{T25} (we only use \eref{strong confinement} for $L = N^{99}$ and alternative bounds would suffice). The second inequality in \eref{confinement result} is a consequence of our assumptions on the potential and estimates on $\mu_\theta$, see for instance (\cite[Lemma 3.5]{AS22}). 

	A key technical input is a novel concentration estimate for the potential generated by fluctuations
	\begin{equation}\label{e.fluct}
		\fluct_N:=\emp_N-\mu_\theta
	\end{equation}
	which may be of interest in its own right.
	
	\begin{proposition}\label{p. concentration}
		Assume $\beta \leq 1$ and \eref{growth}.
		We have
		\begin{equation} 
			\P_{N,\beta}\( \left| \sum_{i=1}^k \h^{\fluct}(y_i)\right| \geq k TN^{-1/2} \) \leq 4k\(e^{-\frac12 TN^{1/2}} + e^{-\beta TN}\),
		\end{equation}
		for any $T \geq C \log N$ for a large enough $C > 0$, uniformly in points $y_1,\ldots,y_k \in \R^2$. Here $N$ is restricted to be large uniformly over selections of $y_1,\ldots,y_k$ in any fixed compact set.
	\end{proposition}
	
	\subsection{Comparison with Literature}
	As discussed above, a seminal study of the microscopic behavior of Coulomb and Riesz gases was undertaken by the authors in \cite{LS15}, where they studied a push-forward of $Q_{\overline{z},N}$ under a certain averaging operator. In that work, they identify a critical temperature scale $\beta \simeq N^{-\frac{\s}{\d}}$ 
    for Riesz gases at which the local behavior of the gas minimizes a competition between a renormalized energy term and an entropic term, phrased as a relative entropy of point processes with respect to a Poisson point process. In the $\s=0$ case, corresponding to the logarithmic interactions, this corresponds to constant temperature. For the one-dimensional log gas, this minimizer is unique (\cite{EHL21}) and is well known to be the $\mathrm{Sine}_\beta$ point process (cf \cite{VV09} and \cite{KS09}) in the bulk and the $\mathrm{Airy}_\beta$ point process \cite{RRV11} at the edge. In the two-dimensional case with $V(x)=|x|^2$ and $\beta=2$, this is the Ginibre point process.
	
	The regime $\beta \simeq N^{-1}$, so-called \textit{high temperature regime}, has seen a lot of attention in recent years. In particular, it can be seen as a crossover between the strong interaction we see at constant $\beta$ and the independent $\beta=0$ behavior (see \cite{ABG12} and \cite{AB19}) and we see convergence of the empirical measures $\emp_N$ to a thermal equilibrium measure $\mu_\theta$ discussed above. The behavior of $Q_{\overline{z},N}$ at this temperature regime has been well-studied in the one-dimensional logarithmic case (\cite{BGP15} and \cite{NT20}) and for general Riesz interactions in \cite{Lam21}. Given the above discussion, one expects entropic effects to dominate for any $\beta \downarrow 0$ and to recover convergence to a Poisson point process (see \cite{AD14}); for the Gaussian $\beta$-ensembles, this was studied in \cite{BGP15}.

    Recently, in \cite{T25}, weak convergence of $Q_{\overline z,N}$ to a homogeneous {\it mixed} Poisson process was proved so long as $\beta_N \to 0$ as $N \to \infty$ at any speed (modulo a subsequence). The argument uses isotropic averaging (to prove that the $k$-point correlation functions are asymptotically subharmonic in any dimension) followed by Liouville's theorem (to show that the $k$-point correlation functions are therefore constant at $N = \infty$ in $\d=2$). Identifying the limiting process as a Poisson process of explicit intensity, rather than a mixed Poisson process, appears to be beyond the reach of the ``soft" argument based on Liouville's theorem.
	
	Our approach is motivated by ideas from \cite{BGP15}, \cite{NT20} and \cite{Lam21}, where the authors consider convergence of the correlation functions of the point process. We are able to extend the temperature regime for the Coulomb gas to $\beta \ll N^{-\frac{1}{2}-\alpha}$ for $\alpha>0$ by a quantitative description of this convergence, as we undertook for weakly interacting particle systems in \cite{PGPT24}. A key technical input is a concentration estimate for the Coulomb potential generated by fluctuations of linear statistics, Proposition \ref{p.concentration}. This is a challenging estimate to obtain for the Coulomb gas due to the singular nature of the logarithm as a test function. Our approach relies on techniques for studying concentration developed in \cite{CHM18} and \cite{GZPG24}, coupled with a regularization of the kernel. 
	
	Fluctuations of linear statistics are interesting in their own right, and have seen a lot of attention both for the one-dimensional log gas (\cite{J98}, \cite{BG13}, \cite{BG16}, \cite{BL18}, \cite{BLS18}, \cite{Lam21b}, \cite{P24} and others) and the two-dimensional Coulomb gas (\cite{RV07}, \cite{LS18}, \cite{BBNY19} and \cite{S22}). Results for singular test functions such as the logarithm are more sparse, but have been studied in \cite{BMP22} for the one-dimensional log gas and \cite{B21} for the Riesz gas on the periodic torus. A key technique that allows us to go from regularized statistics to the logarithm are overcrowding estimates recently developed for general $\beta$ in \cite{T25} . These rely on the powerful tool of isotropic averaging, developed in \cite{T24} and \cite{L17}, which are specific to the Coulomb gas.
	
	The study of fluctuations of linear statistics for the logarithm are interesting, not just for studying the local point process of the gas, but also for the field $\sqrt{\beta}\h^{\fluct}$ itself, which is expected to behave like a log-correlated field. In the case of the $\beta$-ensembles and random normal matrices, this corresponds to the log-characteristic polynomial of the matrix ensemble. The maximum of log-correlated fields is well studied (\cite{DRZ17}), and one expects
	\begin{equation*}
		\max \sqrt{\beta}\h^\fluct =\log N-\frac{3}{4}\log \log N+Z_N
	\end{equation*}
	as $N \rightarrow \infty$, where $Z_N$ is a shifted Gumbel. This conjecture for the log-characteristic polynomial is the Fyodorov-Hiary-Keating conjecture (\cite{FHK12}). Progress has recently been made on understanding the first-order asymptotics for the Coulomb gas in \cite{LLZ24} and \cite{peilen2025maximum} and the one-dimensional log gas in \cite{BLZ23}, but the remaining asymptotics are largely open.
	
	\section{Concentration for Fluctuations of the Logarithmic Potential}
    \subsection{Splitting and the Next-Order Energy}
    	Given the law of large numbers (\ref{eq: LLN}), it is natural to ``split" off the deterministic first-order limit when considering finer statistics of the gas. In particular, one has the so-called ``splitting formula" (see \cite[Lemma 5.2]{S24}, introduced in \cite{SS12})
	\begin{equation}\label{e.splitting}
		\mcl H_N(X_N)=N^2\mcl E_\theta(\mu_\theta)-\frac{N}{\theta}\sum_{i=1}^N \log \mu_\theta(x_i)+N^2\Fenergy_N(X_N,\mu_\theta),
	\end{equation}
	where $\Fenergy_N$ is a next-order energy given by 
	\begin{equation}\label{e.Fenergy}
		\Fenergy_N(X_N,\mu) = \frac{1}{2}\iint_{\Delta^c} \g(x-y) \(\frac{1}{N} \sum_{i=1}^N \delta_{x_i} - \mu\)^{\otimes 2}(dx,dy),
	\end{equation}
	which can be thought of as the energy of a ``jellium" consisting of positive point charges and a neutralizing, negatively charged background measure $\mu$. The set $\Delta = \{x=y\} \subset (\R^\d)^2$ above denotes the diagonal, which is removed to neglect the formally infinite self-energy of a point charge. Using (\ref{e.splitting}), one can equivalently write the Gibbs measure (\ref{e.Gibbs1}) as 
	\begin{equation}\label{e.Gibbs}
		\P_{N,\beta}(dX_N) = \frac{1}{\Kpart_{N,\beta}} e^{-\beta N^2 \Fenergy_N(X_N,\mu_\theta)} \mu_{\theta}^{\otimes N}(dX_N)
	\end{equation}	
	where $\theta = \beta N$, and $\Kpart_{N,\beta}$ is the next-order partition function
	\begin{equation}\label{e.Kpart}
		\Kpart_{N,\beta}(X_N):=\int e^{-\beta N^2 \Fenergy_N(X_N,\mu_\theta)} \mu_{\theta}^{\otimes N}(dX_N).
	\end{equation}
	We take (\ref{e.Gibbs}) as our preferred representation of the Coulomb gas in the remainder of this paper. The next-order partition function is conveniently bounded below, which we state here for reference.
	\begin{proposition}(\cite[Prop. 5.10]{PG23})\label{p.Kpartlower}
		Let $\Kpart_{N,\beta}$ be as in (\ref{e.Kpart}). Then $\Kpart_{N,\beta} \geq 1$.
	\end{proposition}

\subsection{Method of Proof}
The remainder of this section is devoted to the proof of Proposition \ref{p. concentration}. The approach is similar in spirit to \cite{PGPT24}, with some technical modifications due to the lack of integrability of the Fourier transform of the kernel $\g$. To circumvent this issue, we smear our point charges by replacing point masses $\delta_{x_i}$ with smoothed versions $\delta_{x_i}^\eta$, uniform measures on the boundary of the ball of radius $\eta>0$ centered at $x_i$. This regularization procedure, first introduced for the Coulomb gas in \cite{SS15-2} (see \cite[Chapter 4]{S24} for a discussion) has been used extensively in the study of Coulomb gases to regularize the kernel $\g$.
	
	Our approach is slightly different than previous works, in that we will regularize both the point charges and the background measure $\mu_\theta$. More specifically, letting $\phi_\eta$ denote the uniform measure on the boundary of the ball of radius $\eta$ centered at zero, we will consider 
	\begin{equation*}
		\fluct \ast \phi_\eta:=\left(\emp_N-\mu_\theta\right)\ast \phi_\eta
	\end{equation*}
	instead of just the point regularization $\emp_N \ast \phi_\eta -\mu_\theta$. This modification allows us to argue use Fourier transform representations to argue that $\mcl E(\fluct \ast \phi_\eta)$ is large whenever $\h^{\fluct \ast \phi_\eta}$ is large at a fixed point, where $\mcl E$ is as in (\ref{e.E}). We will argue as well that $\mcl E(\fluct \ast \phi_\eta)$ is a near lower bound for for $\Fenergy_N(X_N,\mu_\theta)$, up to certain ``renormalization" errors, which will then allow us to obtain concentration for the smoothed potential $\h^{\fluct \ast \phi_\eta}$ via estimates inspired by those in \cite{CHM18}. Finally, we show that $\h^{\fluct \ast \phi_\eta}$ is pointwise close to $\h^{\fluct}$ in exponential moments so long as $\eta$ is sufficiently small using isotropic averaging techniques.
	
	\subsection{Lower bounds for the energy}
    The key to this concentration estimate is that measures with a large Coulomb potential at a given point have a correspondingly large Coulomb energy. We will use in our argument that the Fourier transform of the surface probability measure on the $(\d-1)$-sphere is explicitly computable, namely 
	\begin{equation}\label{e.FT}
		\hat{\phi}_1(\xi)=\|\xi\|^{1-\frac{\d}{2}}J_{\frac{\d}{2}-1}(\|\xi\|)
	\end{equation}
	where $J_a$ is a Bessel function of the first kind (see \cite[p. 154]{SW71}). Importantly, these functions are bounded for real arguments, and have asymptotic behavior (\cite[p. 364]{AS74})
	\begin{equation*}
		|J_a(r)| \lesssim \frac{1}{\sqrt{r}}
	\end{equation*}
	as $r \rightarrow \infty$.
	\begin{lemma} \label{l.minE}
		Let $\phi_{\eta}$ denote the uniform measure on the boundary of the ball of radius $\eta$ centered at zero with total mass $1$. For any $\ep, \delta, \eta > 0$, let $ \mcl D_{\ep,\delta,\eta}$
        denote the following set of (signed) measures:
		\begin{equation*}
		\mcl D_{\ep,\delta,\eta}:=\left\{\nu \in L^\infty(\R^2) \ : \ \nu(\R^2) = 0, \int |x| |\nu|(dx) \leq \delta^{-1}, |\h^{\nu \ast \phi_\eta}(0)| \geq \ep\right\}.
		\end{equation*}
		Then   		
        \begin{equation} 
			\inf_{\nu \in \mcl D_{\ep,\delta, \eta}} \mcl E(\nu) \geq \frac{ \c_0 \ep^2}{1 + \log \ep^{-1} + \log \delta^{-1} + \log \eta^{-1}}
		\end{equation} for some $\c_0 > 0$.
	\end{lemma}
	\begin{proof}
		Suppose $\nu \in \mcl D_{\ep,\delta, \eta}$. Since $\nu/\ep \in \mcl D_{1,\ep \delta, \eta}$ and $\mcl E(\nu/\ep) = \ep^{-2}\mcl E(\nu)$, we may assume without loss of generality that $\ep = 1$. Note that we have the representations 
		\begin{equation}
        \begin{split}
			\h^{\nu \ast \phi_\eta}(0) &= \int \hat{\g}(\xi) \hat \nu(\xi) \hat \phi_\eta(\xi) d\xi, \\
			\mcl E(\nu) &= \frac12 \int \hat{\g}(\xi) |\hat \nu(\xi)|^2 d\xi,
            \end{split}
		\end{equation}
		where the above should be understood at first in a principle value sense. We would like to argue via Cauchy-Schwarz that
		\begin{equation*}
		1 \leq |\h^{\nu \ast \phi_\eta}(0)| \leq \(\int \hat \g(\xi) |\hat \phi_\eta(\xi)|^2 \)^{1/2} \sqrt{2 \mcl E(\nu)},
		\end{equation*}
		except the integral on the RHS is infinite due to the non-integrability of $\hat \g$ at $0$. So, we will first apply a cutoff at the scale $r = C^{-1} \delta$ for a sufficiently large constant $C$. Since $\nu(\R^2) = 0$ and $\nu$ has a first moment bounded by $\delta^{-1}$, we can estimate
		\begin{equation*}
			|\hat \nu(\xi)|= C\left|\int \left(e^{-2\pi i x \cdot \xi}-1\right)\nu(x)~dx\right|    \leq C|\xi|\int |x||\nu(x)|~dx \leq C\delta^{-1}|\xi|,
		\end{equation*}
	and therefore, for $B_r$ a ball of radius $r>0$ centered at zero, we have
		\begin{equation} \label{e.lowfreq}
		\left| \int_{B_r} \hat \g (\xi)\hat \nu(\xi) \hat \phi_{\eta}(\xi) d\xi \right | \leq C \delta^{-1} \int_{B_r} |\xi| |\hat \g(\xi)||\hat \phi_\eta( \xi)|d\xi \leq C\delta^{-1} \int_{B_r} \frac{1}{|\xi|}d\xi \leq C \delta^{-1} r \leq \frac12.
	\end{equation}
	Here we used that $|\hat \phi_\eta| \leq 1$ and $\hat \g(\xi) = \frac{1}{2\pi |\xi|^2}$.
		
		Away from this truncation we can proceed as before. By Cauchy-Schwarz, we may bound
		$$
		\left| \int_{B_r^c} \hat \g (\xi)\hat \nu(\xi) \hat \phi_{\eta}(\xi) d\xi \right | \leq \(\int_{B_r^c} \hat \g (\xi)  |\hat \phi_{\eta}(\xi)|^2 d\xi \)^{1/2} \sqrt{2 \mcl E(\nu)}.
		$$
		
		Looking at the integral on the RHS, one has for the intermediate frequencies that
		\begin{equation*}
			\int_{B_{\eta^{-1}} \cap B_r^c} \hat \g(\xi) |\hat \phi_\eta(\xi)|^2 d\xi \leq C\int_{B_{\eta^{-1}}\cap B_r^c}  \frac{1}{|\xi|^{2}} d \xi \leq C\left( 1 + \log \eta^{-1} + \log \delta^{-1}\right),
		\end{equation*}
		and for the high frequencies one has (via the substitution $u = \eta \xi$) 
		\begin{align*}
			\int_{B_{\eta^{-1}}^c} \hat \g(\xi) |\hat \phi_\eta(\xi)|^2 d\xi &\leq C\int_{B_{\eta^{-1}}^c} \frac{|\hat \phi_1(\eta \xi)|^2}{|\xi|^2} d\xi \\
			&= C  \int_{B_{1}^c} \frac{|\hat \phi_1(u)|^2}{|u|^2} du \\
            &\leq C \int_{B_1^c}\frac{1}{|u|^{3}}~du\\
            &\leq C,
		\end{align*}
		making use of (\ref{e.FT}) and the corresponding Bessel function asymptotics. 	Combining with \eref{lowfreq}, we have proved
		$$
		1 \leq \left| \int_{\R^2} \hat \g (\xi)\hat \nu(\xi) \hat \phi_{\eta}(\xi) d\xi \right | \leq C(1 + \log \eta^{-1} + \log \delta^{-1})^{1/2}\sqrt{\mcl E(\nu)}+ \frac{1}{2}.
		$$
		The lemma follows by rearranging the above inequality. 
	\end{proof}

\begin{remark}
	One can prove that in $\d \geq 2$ we have
	\begin{equation} 
		\inf_{\nu \in \mcl D_{\ep,\delta, \eta}} \mcl E(\nu) \geq \frac{ \c_0 \ep^2\eta^{\d-2}}{1 + \log \ep^{-1} + \log \delta^{-1} + \log \eta^{-1}}.
	\end{equation}
	Notice that the RHS degenerates more rapidly for $\eta \downarrow 0$ in $\d \geq 3$ than in $\d=2$. This is the main reason that \tref{1} is restricted to $\d=2$. 
\end{remark}
	
	\subsection{Approximation}
	This energy $\mathcal{E}$ can be used to obtain lower bounds for the next-order energy $\Fenergy_N$ via the following renormalization bound. We consider general $\d \geq 2$ since the proof is identical.
	\begin{lemma} \label{l.regularization}
		Let $\fluct = \emp_N - \mu_\theta$, and let $\phi_\eta$ be the uniform probability measure on the boundary of the ball 
        of radius $\eta$. Then,
		\begin{equation}
			\mcl E((\emp_N - \mu_\theta) \ast \phi_\eta) - \Fenergy(X_N, \mu_\theta) \leq \frac{\g(\eta)}{2N} +C \left(\| \mu_\theta \|_{L^\infty(\R^\d)}+\| \mu_\theta \|_{L^\infty(\R^\d)}^2\right) \eta^2.
		\end{equation}
	\end{lemma}
	\begin{proof}
		We will bound the two quantities
		$$
		\mcl E(\emp_N \ast \phi_\eta - \mu_\theta) - \Fenergy(X_N, \mu_\theta) \quad \text{and} \quad \mcl E((\emp_N - \mu_\theta) \ast \phi_\eta) -	\mcl E(\emp_N \ast \phi_\eta - \mu_\theta) 
		$$
		separately.

		\textbf{Step 1.} Estimate for $\mcl E(\emp_N \ast \phi_\eta - \mu_\theta) - \Fenergy(X_N, \mu_\theta)$.

            First, we compute
		\begin{align*}
			&\mcl E(\phi_\eta \ast \emp_N - \mu_\theta) - \Fenergy_N(X_N,\mu_\theta) \\
            = &\iint \g(x - y) (\emp_N - \phi_\eta \ast \emp_N)(dx) \mu_\theta(dy)\\
            &\ \ \ \ \ \ +\frac{1}{2N^2}\sum_{i=1}^N \iint \g(x-y)~\left(\phi_\eta \ast \delta_{x_i}\right)^{\otimes 2}(dx,dy)\\
             &\ \ \ \ \ \ +\frac{1}{2N^2}\sum_{i \ne j}\left(\iint \g(x-y)~\left(\phi_\eta \ast \delta_{x_i}\right)(dx)\left(\phi_\eta \ast \delta_{x_j}\right)(dy)-\g(x_i-x_j)\right)
		\end{align*}
		and standard superharmonicity arguments tell us that the last term is negative and the middle term is explicitly computable (see \cite[\S 4.1]{S24}), yielding
		$$
		\mcl E(\phi_\eta \ast \emp_N - \mu_\theta) - \Fenergy_N(X_N,\mu_\theta) \leq \frac{1}{2N} \g(\eta) + \iint \g(x - y) (\emp_N - \phi_\eta \ast \emp_N)(dx) \mu_\theta(dy).
		$$
		By associativity of convolution we have
		\begin{equation*} 
			\iint \g(x - y) (\emp_N - \phi_\eta \ast \emp_N)(dx) \mu_\theta(dy) = \iint (\g - \phi_\eta \ast \g)(x - y) \emp_N(dx) \mu_\theta(dy). 
		\end{equation*}
		By Young's convolution inequality,
		\begin{equation}\label{e.reg.background}
        \begin{split}
			\iint (\g - \phi_\eta \ast \g)(x - y) \emp_N(dx) \mu_\theta(dy) &\leq \|(\g-\phi_\eta \ast \g)\ast \mu_\theta\|_{L^\infty(\R^\d)}\\
			&\leq \|\g-\phi_\eta \ast \g\|_{L^1(\R^\d)}\|\mu_\theta\|_{L^\infty(\R^\d)} \\
            &\leq C\|\mu_\theta\|_{L^\infty(\R^\d)}\eta^2
            \end{split}
		\end{equation}
		where we have used that $\g-\g \ast \phi_\eta$ has $L^1$ norm of order $\eta^2$ (\cite[Lemma 4.9]{S24}).

        \textbf{Step 2.} Estimate for $\mcl E((\emp_N - \mu_\theta) \ast \phi_\eta) - \mcl E(\emp_N \ast \phi_\eta - \mu_\theta)$.

		We start by writing
		\begin{align*}
			&\mcl E((\emp_N - \mu_\theta) \ast \phi_\eta) - \mcl E(\emp_N \ast \phi_\eta - \mu_\theta) 
			\\
			= &\iint \g(x-y) (\phi_\eta \ast \emp_N- \mu_\theta)(dx)(\mu_\theta - \phi_\eta \ast \mu_\theta)(dy) + \mcl E(\mu_\theta - \phi_\eta \ast \mu_\theta).
		\end{align*}
		The first term on the right hand side is equal by associativity of convolution to 
		$$
		\iint (\g - \phi_\eta \ast \g)(x-y) (\phi_\eta \ast \emp_N - \mu_\theta)(dx) \mu_\theta(dy)
		$$
		which is bounded by $C \eta^2 \| \mu_\theta \|_{L^\infty}^2$ in the same way as \eref{reg.background}. The second term on the right hand side can also be written as 
		\begin{align*}
			&\lefteqn{\frac{1}{2}\iint \g(x-y) \left(\mu_\theta - \mu_\theta \ast \phi_\eta\right)(dy)\left(\mu_\theta - \mu_\theta \ast \phi_\eta\right)(dx)} \quad & \\  
			= & \frac{1}{2}\iint (\g-\phi_\eta \ast \g)(x-y)\mu_\theta(dy)\left(\mu_\theta - \mu_\theta \ast \phi_\eta\right)(dx),
		\end{align*}
		which is similarly bounded by $C \eta^2 \| \mu_\theta \|_{L^\infty}^2$, establishing the result.
	\end{proof}
	
	Before we turn to establishing the desired concentration estimate, we first need to show that the potential generated by smeared points (which is amenable to the energy estimates discussed in Lemma \ref{l.minE}) is sufficiently close to the potential generated by the actual point configuration at a fixed point $x$. This difference is large only if there is a point very close to $x$, which is a low probability event, but the estimate requires some care due to the logarithmic singularity in the potential. We make use of isotropic averaging techniques introduced and developed in \cite{L17}, \cite{T24} and \cite{T25}.
	
	We let $\psi : \R^2 \to \R$ be a nonnegative, rotationally symmetric mollifier of total mass $1$ supported within $\overline{B_1(0)}$. We let $\psi_\eta(x) = \eta^{-2}\psi(x/\eta)$. In practice, we will set $\psi = \phi_{1/2} \ast \phi_{1/2}$ (and so $\psi_{2\eta} = \phi_{\eta} \ast \phi_{\eta}$).
	\begin{proposition} \label{p.reg.h.change}
		Assume \eref{growth}. Let $\psi_\eta$ be as above. One has for all $t \in (0,2)$, $\beta \leq 1$, and $\eta \leq  N^{-1/2}\beta^{-1/2}$ that
		\begin{equation} \label{e.micro.smoothing}
			\E_{N,\beta} \left[ e^{ t N \(\h^{\fluct}(x) - \h^{\fluct \ast \psi_\eta}(x)\)} \right] \leq \exp\(\frac{CN \eta^{2}}{2-t} \)
		\end{equation}
		for a constant $C$ depending only on $\sup_{y : |y-x| \leq 1} |\max(\Delta V(y),0)|$. One also has
		\begin{equation} \label{e.micro.smoothing.lower}
			\h^{\fluct}(x) - \h^{\fluct \ast \psi_\eta}(x) \geq -C \| \mu_{\theta} \|_{L^\infty} \eta^2.
		\end{equation}
	\end{proposition}
	\begin{proof}
		\textbf{Step 1. }Background results.
        
        We begin by stating a couple results that are consequences of the isotropic averaging technique. All implicit constants will be independent of $x$ due to the bound $\Delta V \leq C$ from \eref{growth}. From \cite[Theorem 1]{T24}, one has for any $R \geq N^{-1/2}$ that
		\begin{equation} \label{e.overcrowd}
			\P_{N,\beta}(|\{x_i \in B_R(x)\}| \geq Q ) \leq e^{-\frac12 \beta Q^2} \quad \forall Q \geq C \beta^{-1} + CNR^2
		\end{equation}
		for $C$ depending only on $\sup_{y : |y-x| \leq 1} |\max(\Delta V(y),0)|$. We will consider $\eta \leq R = N^{-1/2} \beta^{-1/2}$ below.
		
		Let $\tilde{\rho}_k(\cdot | x_{k+1},\ldots,x_M)$ be the Lebesgue density of the law of $x_1,\ldots,x_k$ conditioned on $x_{k+1}, \ldots,x_N$. In \cite[Proposition 3.1, specifically by iterating (3.6)]{T25} it is proved that a.s.\
		\begin{equation} \label{e.iso.upper}
        \begin{split}
			&\tilde{\rho}_k(y_1,\ldots,y_k | x_{k+1},\ldots,x_N) \\
            \leq & \frac{C^ke^{C \beta N k R^2}}{R^{2k}} \int_{B_{R}(y_1) \times \cdots \times B_R(y_k)} \tilde{\rho}_k(z_1,\ldots,z_k | x_{k+1},\ldots,x_N) dz_1 \cdots dz_k,
            \end{split}
		\end{equation}
		for a constant $C$ depending only on $\sup_{i=1,\ldots,k} \sup_{y : |y-y_i|\leq 1} |\max(\Delta V(y),0)|$.

        \textbf{Step 2. }Separate $B_\eta(x)$ and $B_\eta(x)^c$.
        
		Now we are ready to prove \eref{micro.smoothing}. Splitting the integral according to which particles fall within $B_\eta(x)$, one has
		\begin{equation} \label{e.smear.kdecomp}
        \begin{split}
			&\E\left[ e^{ t N \(\h^{\fluct}(x) - \h^{\fluct \ast \psi_\eta}(x)\)} \right] \\
            = &\sum_{k=0}^N \binom{N}{k}  \int_{B_\eta(x)^k \times (\R^2 \setminus B_\eta(x))^{N-k}} e^{ t N \(\h^{\fluct}(x) - \h^{\fluct \ast \psi_\eta}(x)\)} 	\P_{N,\beta}(dX_N).
            \end{split}
		\end{equation}
		Note that $\h^{\fluct}(x) - \h^{\fluct \ast \psi_\eta}(x)$ is measurable with respect to the $\sigma$-algebra generated by the locations of the particles within $B_\eta(x)$. This is because $\g - \g \ast \psi_\eta$ is supported within $B_\eta(0)$, where it is equal to $x \mapsto -\log(|x|/\eta)$.
		
\textbf{Step 3. }Conditioning on particles in $B_\eta(x)^c$.

We abbreviate $X_k = (x_1,\ldots,x_k)$, $Z_k = (z_1,\ldots,z_k)$, and $\tilde \rho_k(X_k) = \tilde \rho_k(x_1,\ldots,x_k | x_{k+1},\ldots,x_N)$. Applying \eref{iso.upper}, we see a.s.\ in $x_{k+1},\ldots,x_N$ that
		\begin{equation} \label{e.disint}
        \begin{split}
			&\lefteqn{\int_{B_\eta(x)^k} e^{ t N \(\h^{\fluct}(x) - \h^{\fluct \ast \psi_\eta}(x)\)} \tilde{\rho}_k(X_k) dX_k}  \\   \leq &\frac{C^ke^{C \beta N R^2 k}}{ R^{2k}} \int_{B_\eta(x)^k}\int_{B_R(x_1) \times \cdots \times B_R(x_k)} e^{ t N \(\h^{\fluct}(x) - \h^{\fluct \ast \psi_\eta}(x)\)} \tilde{\rho}_k(Z_k) dZ_k dX_k \\ 
            \leq &\frac{C^ke^{C \beta N R^2 k}}{R^{2k}} \( \int_{B_{R+\eta}(x)^k} \tilde{\rho}_k(Z_k) dZ_k \) \( \int_{B_\eta(x)^k} e^{ t N \(\h^{\fluct}(x) - \h^{\fluct \ast \psi_\eta}(x)\)}  dX_k \),
            \end{split}
		\end{equation}
		where in the last line we enlarged an integration region from $B_R(x_1) \times \cdots \times B_R(x_k)$ to $B_{R+\eta}(x)^k$.
		
		We now examine $\g - \g \ast \psi_\eta$. Using the convention $f_\eta(x) = \eta^{-2} f(x/\eta)$, from which one derives $\g_{\eta^{-1}} = \eta^2 (\g + \log \eta)$, one has
		$$
		\g - \g \ast \psi_\eta = (\g_{\eta^{-1}} - \g_{\eta^{-1}} \ast \psi_1)_\eta = \eta^2 (\g - \g \ast \psi_1)_\eta.
		$$
		Note $\g - \g \ast \psi_1$ is supported within $B_1(0)$ by harmonicity of $\g$ away from $0$. Within $\overline{B_1(0)}$, we have $\g - \g \ast \psi_1 \leq \g$. It follows that $\g - \g \ast \psi_\eta$ is supported within $\overline{B_\eta(0)}$ and bounded by $-\log(|x|/\eta)$ within its support.
		
		By superharmonicity of $\g$ and the above, we have
		\begin{align*}
			\int_{B_\eta(x)^k} e^{ t N \(\h^{\fluct}(x) - \h^{\fluct \ast \psi_\eta}(x)\)}  dX_k &\leq \int_{B_\eta(x)^k} e^{ t  \(\sum_{i=1}^k \g(x_i - x) - \g \ast \psi_\eta(x_i - x)\)}  dX_k \\ 
			&\leq  \(\int_{B_\eta(x)} \frac{\eta^t}{|x-y|^{t}} dy\)^k\\
            &\leq \(\frac{C\eta^{2}}{2-t}\)^k.
		\end{align*}
	
		Now, integrating \eref{disint} against the law of $x_{k+1},\ldots,x_N$ over $(\R^2 \setminus B_{\eta(x)})^{N-k}$ yields
		\begin{equation} \label{e.smear.singlek}
        \begin{split}
			 &\int_{B_{\eta}(x)^k \times (\R^2 \setminus B_{\eta}(x))^{N-k}} e^{ t N \(\h^{\fluct}(x) - \h^{\fluct \ast \psi_\eta}(x)\)} \P_{N,\beta}(dX_N)  \\
            \leq &\(\frac{C e^{C\beta N R^2} \eta^{2}}{(2-t) R^2}\)^k \P_{N,\beta}\(x_1,\ldots,x_k \in B_{R+\eta}(x)\).
            \end{split}
		\end{equation}

        \textbf{Step 4. } Proof of \eqref{e.micro.smoothing}.
        
		We use exchangeability and \eref{overcrowd} to estimate the last probability in \eref{smear.singlek} as
		\begin{align*}
			& \P_{N,\beta}\(x_1,\ldots,x_k \in B_{R+\eta}(x)\)  \\
             = &\sum_{m=k}^\infty     \frac{\binom{N-k}{m-k}}{\binom{N}{m}} \P_{N,\beta}\(|\{x_i : x_i \in B_{R+\eta}(x) \}| = m\)  \\
			 = &\binom{N}{k}^{-1} \sum_{m=k}^\infty \binom{m}{k} \P_{N,\beta}\(|\{x_i : x_i \in B_{R+\eta}(x) \}| = m\) \\
			\leq &\binom{N}{k}^{-1} \frac{1}{k!}  \sum_{m = k}^\infty m^k \P_{N,\beta}\(|\{x_i : x_i \in B_{R+\eta}(x) \}| = m\)  \\
			 \leq &\binom{N}{k}^{-1} \frac{1}{k!} \( (C\beta^{-1})^{k} \P_{N,\beta}\(|\{x_i : x_i \in B_{R+\eta}(x) \}| \leq C\beta^{-1}\) + \sum_{m = \max(k,C\beta^{-1})}^\infty m^k e^{-\frac12 \beta m^2} \).
		\end{align*}
		The ratio of successive terms in the last sum is
		$$
		\(\frac{m+1}{m}\)^k e^{-\beta m} \leq e^{\frac{k}{m} - \beta m},
		$$
		which is less than $1/2$ so long as $C$ is large enough. We conclude that
		$$
		\P_{N,\beta}\(x_1,\ldots,x_k \in B_{R+\eta}(x)\) \leq (C\beta^{-1})^{k} \binom{N}{k}^{-1} \frac{1}{k!},
		$$
		and combining with \eref{smear.singlek} and \eref{smear.kdecomp} shows
		$$
		\E\left[ e^{ t N \(\h^{\fluct}(x) - \h^{\fluct \ast \psi_\eta}(x)\)} \right] \leq  \sum_{k=0}^\infty \frac{1}{k!} \(\frac{C e^{C\beta N R^2} \eta^{2}}{\beta (2-t) R^2}\)^k = \exp \(\frac{C N\eta^{2}}{2-t}\),
		$$
		where we substituted $R = N^{-1/2} \beta^{-1/2}$. This finishes the proof of \eref{micro.smoothing}.
        
       \textbf{Step 5. } Proof of equation \eqref{e.micro.smoothing.lower}.

       While $\h^{\fluct}-\h^{\fluct \ast \psi_\eta}$ can only be bounded from above in probability, we have a deterministic lower bound. Equation \eref{micro.smoothing.lower} follows from the computation
		$$
		\h^{\fluct}(x) - \h^{\fluct \ast \psi_\eta}(x) \geq \h^{\mu_\theta \ast \psi_\eta}(x) - \h^{\mu_\theta}(x) = (\g \ast \psi_\eta - \g) \ast \mu_\theta.
		$$
		Young's inequality finishes the proof once one notes that $\g \ast \psi_\eta - \g$ has $L^1$ norm of order $\eta^2$.
	\end{proof}
	
	\subsection{Concentration}	We are now in a position to obtain our main concentration result.
	
	\begin{proposition} \label{p.concentration}
		Assume $\beta \leq 1$ and \eref{growth}.
		We have
		\begin{equation}
			\P_{N,\beta}\( \left| \sum_{i=1}^k \h^{\fluct}(y_i)\right| \geq k TN^{-1/2} \) \leq 4k\(e^{-\frac12 TN^{1/2}} + e^{-\beta TN}\),
		\end{equation}
		for any $T \geq C \log N$ for a large enough $C > 0$, uniformly in points $y_1,\ldots,y_k \in \R^2$. Here $N$ is restricted to be large uniformly over selections of $y_1,\ldots,y_k$ in any fixed compact set.
	\end{proposition}
	
	\begin{proof}

    \textbf{Step 1.} First simplification. 
    
		First, it is enough to bound
		\begin{equation*}
			\P_{N,\beta}\(\left| \h^{\fluct}(y_1)\right| \geq  TN^{-1/2} \) \leq 4\(e^{-\frac12 TN^{1/2}} + e^{-\beta TN}\),
		\end{equation*}
		by a pigeonhole argument in $k$ and a union bound. We will also let $y_1 = 0$ for simplicity. Let $\ep = \frac{1}{2}TN^{-1/2}$, and $\mcl G_{\pm}({\lambda},\nu) = \{X_N : \pm \h^{\nu}(0) \geq  \lambda \}$ for any $\nu$ and $\lambda > 0$. Define $\eta = N^{- 100}$, and let $\psi_{2\eta} = \phi_\eta \ast \phi_\eta$.
		
	\textbf{Step 2.} Now we estimate the probability that $\h^{\fluct} (0)$ is large but $\h^{\fluct \ast \psi_{2\eta}}(0)$ 
    is not.
    
    We need to estimate
		\begin{equation*}
			\P_{N,\beta}\(\mcl G_+(2\ep,\fluct) \cap \(\mcl G_+(\ep,\fluct \ast \psi_{2\eta})\)^c \) \leq \P_{N,\beta}(\h^{\fluct}(0) - \h^{\fluct \ast \psi_{2\eta}}(0) \geq  \ep).
		\end{equation*}
		The Chebyshev inequality coupled with Proposition \ref{p.reg.h.change} (with $t=1$) yields 
		\begin{equation*}
			\P_{N,\beta}\left(\h^{\fluct}(0) - \h^{\fluct \ast \psi_{2\eta}}(0) \geq  \ep\right) \leq e^{-N \epsilon+CN\eta^2} \leq 2 e^{-\frac12 TN^{1/2}}.
		\end{equation*}
		
	\textbf{Step 3.} Next, we turn to estimating the probability that $\h^{\fluct \ast \psi_{2\eta}}(0)$ exceeds $\ep$. We will divide into two cases: $T \leq N$ and $T > N$. 

    \textbf{Substep 3.1} We first consider $T \leq N$.
     
    Abbreviate $\mcl G_+ = \mcl G_+(\ep, \fluct \ast \psi_{2\eta}) \cap \{\int |x| |\phi_\eta \ast \fluct|(dx) \leq \delta^{-1}\}$ for $\delta^{-1} = N^{100}$, 
    and note first by \lref{regularization} we have
		\begin{equation} \label{e.reg.app}
			\int_{\mcl G_+}e^{-\beta N^2 \Fenergy_N(X_N,\mu_\theta)}\mu_{\theta}^{\otimes N}(dX_N) \leq e^{\beta N^2 \Error_1(\eta)} \int_{\mcl G_+} e^{-\beta N^2 \mcl E(\phi_\eta \ast \fluct)} \mu_\theta^{\otimes N}(dX_N)
		\end{equation}
		for 
		$$
		\Error_1(\eta) \leq (2N)^{-1} \g(\eta) + C
		 \left(\|\mu_\theta\|_{L^\infty(\R^\d)}+\|\mu_\theta\|_{L^\infty(\R^\d)}^2\right) \eta^2 \leq \frac{C\log N}{N}.$$
		 By \lref{minE} and recalling $\psi_{2\eta} = \phi_\eta \ast \phi_\eta$, we have
		\begin{align*}
		\mcl G_+ &\subset \left\{\mcl E(\fluct \ast \phi_\eta) \geq \frac{\c_0 \ep^2}{1 + \log \ep^{-1} + \log \delta^{-1} + \log \eta^{-1} } \right\} \\ &\subset \left\{\mcl E(\fluct \ast \phi_\eta) \geq \frac{T^2}{CN \log N } \right\}.
		\end{align*}
		Note further that $\frac{T^2}{CN \log N} \geq \frac{CT}{N}$. Assembling the above into \eref{reg.app}, we find
		\begin{align*}
			\int_{\mcl G_+} e^{-\beta N^2 \Fenergy(X_N,\mu_\theta)}\mu_{\theta}^{\otimes N}(dX_N) &\leq e^{C\beta N\log N - C \beta T N } \leq e^{-\beta TN }.
		\end{align*}
		\pref{Kpartlower} tells us that $K_{N,\beta} \geq 1$, so we can compute
		\begin{equation*}
			\P_{N,\beta}\(\mcl G_+\) \leq e^{-\beta TN }.
		\end{equation*}
		We have proved
        \begin{equation}\label{e.half concentration}
			\P_{N,\beta}(\mcl G_+(2\ep, \fluct)) \leq e^{-\frac12 TN^{1/2}}  + e^{-\beta T N} + \P_{N,\beta}\left(\int |x| |\phi_\eta \ast \fluct|(dx) > N^{100} \right).
		\end{equation}
		It remains to estimate the probability of the event that $\int |x||\phi_\eta \ast \fluct|(dx)>N^{100}$. Clearly, we have
		$$
		\int |x||\phi_\eta \ast \fluct|(dx) \leq \int (|x|+\eta) |\fluct|(dx) =  \int |x| |\fluct|(dx) + 2\eta,
		$$
		and if all particles are within a distance of $N^{99}$ from the origin then this quantity is bounded by $N^{100}$. It is easy to show that this event is extremely likely; say by using \eref{strong confinement} from \cite[Theorem 3]{T25}, that $\zeta = V + \h^{\mu_V} \geq V + \g - C$ for large inputs, and \eref{conf} to see
		\begin{align*}
			\P_{N,\beta} \left(\exists ~|x_i| \geq N^{99} \right) &\leq CN \int_{|x| \geq N^{99}} e^{-\beta N (V(x) - \log |x| - C)}dx \\
			&\leq CN \inf_{|x| \geq N^{99}} e^{-\frac{\beta N}{2}(V(x) - \log |x| - 2C)} \int_{|x| \geq N^{99}} e^{-\frac{\theta_\ast}{2} (V(x) - \log |x|)}dx \\
			&\leq \inf_{|x| \geq N^{99}} e^{-\frac{\beta N}{2}(V(x) - \log |x| - 2C) + \log N}.
		\end{align*}
		We use our assumption \eref{growth} on $V$ to bound the above by $e^{-\beta N^2} \leq e^{-\beta N T}$ since $T \leq N$. We note that, with more complicated error terms, one could treat $V$ with growth only sufficiently super-logarithmic. Inserting this bound into \eref{half concentration} finishes substep 3.1.
		
	\textbf{Substep 3.2} We now consider the case $T \geq N$.
    
    We start by bounding the probability of $\mcl G_+(\ep, \fluct \ast \psi_{2\eta})$. Note that 
		$$
		-\log \max_i |x_i| - C \leq \h^{\fluct \ast \psi_{2\eta}}(0) \leq -\log (C \eta) \leq C \log N,
		$$
		 so it is sufficient to bound the probability that $\log \max_i |x_i| \geq TN^{-1/2} - C$. Arguing as before, one can bound this probability as
		$$
		\inf_{|x| \geq \exp(TN^{-1/2} - C)} e^{-\frac{\beta N}{2}(V(x) - \log |x| - 2C) + \log N} \leq \exp\(-\frac{\beta N}{C} e^{\frac{1}{99} TN^{-1/2}} \) \leq e^{-\beta N T}.
		$$
		
	\textbf{Step 4.}  Estimate for $\P_{N,\beta}(\mcl G_-(100\ep, \fluct))$. 	
    
    The estimate for $\P_{N,\beta}(\mcl G_-(100\ep, \fluct))$ is analogous, except that we estimate more easily
		\begin{align*}
			&\lefteqn{  \P_{N,\beta}\(\mcl G_-(2\ep,\fluct) \cap \(\mcl G_-(\ep,\fluct \ast \psi_{2\eta})\)^c \) }  \\
             \leq & \P_{N,\beta}(\h^{\fluct}(0) - \h^{\fluct \ast \psi_{2\eta}}(0) \leq -50 \ep)\\
             =& 0,
		\end{align*}
		since $\h^{\fluct}-\h^{\fluct \ast \psi_{2\eta}} \geq -C\|\mu_\theta\|_{L^\infty}\eta^2$ by Proposition \ref{p.reg.h.change}.		
	\end{proof}
		This completes our discussion of Proposition \ref{p. concentration}; in the following section, we will apply this to our consideration of the microscopic point process.	
        
	\section{Poisson Convergence}
    \label{sect:convergence}
	In this section we use Proposition \ref{p. concentration} to establish Theorem \ref{t.1}. Let us start by reviewing some preliminary information about point processes and their convergence.
	\subsection{Preliminaries}
	This quick summary is based on material contained in \cite[Appendix A]{Lam21} and \cite[$\S 7$]{PGPT24}. First, recall that a generic point process $\Xi$ is a probability measure on locally finite point configurations in $\R^\d$; it can be easier to think of these as probability measures on the space of Radon measures that take the form $\sum_{p \in P}\delta_p$ for countable subsets $P \subset \R^\d$ with no accumulation points. The law of such a process is determined by its Laplace functional
	\begin{equation}\label{e.Laplace}
		\psi(f):=\E_\Xi \left[e^{-\sum_{p \in P}f(p)}\right]
	\end{equation}
	for all Borel-measurable $f:\R^\d \rightarrow [0,\infty)$. 
	
	\begin{defi}\label{d. weak conv}
		Let $\{\Xi_N\}$ be a sequence of point processes with Laplace functionals $\psi_N$, and $\Xi_\infty$ a point process with Laplace function $\psi_\infty$. We say that $\Xi_N \rightarrow \Xi_\infty$ \textit{weakly} if 
		\begin{equation*}
			\lim_{N \rightarrow \infty}\psi_N(f)=\psi_\infty(f)
		\end{equation*}
		for all continuous and compactly supported $f:\R^\d \rightarrow [0,+\infty)$.
	\end{defi}
	A convenient means of studying weak convergence is through correlation functions. The \textit{correlation functions} $R_k$ are associated via
	\begin{equation}\label{e.corr}
		\psi(f):= 1 + \sum_{i=1}^{\infty} \frac{1}{k!} \int_{\R^\d} \prod_{i=1}^{k} \left( \exp(- f(x_{i})) -1 \right) R_{k}(dx_{1},...dx_{k}).\end{equation}
	Intuitively, one can think of $R_k$ as as the probability density for finding $k$ particles in a given set. Indeed, recall that a homogeneous Poisson point process of intensity $\lambda>0$ with respect to Lebesgue reference measure is uniquely characterized by 
	\begin{enumerate}
		\item The number of points in two disjoint sets are independent, and 
		\item $\# \text{ of points in }\Omega \sim \mathrm{Poisson}(\lambda \mathrm{Leb}(\Omega))$.
	\end{enumerate}
	The correlation functions of such a point process are just $R_k=\lambda^k$.

	The following proposition guides our approach to understanding the local point process $Q_{\overline{z},N}$ given in (\ref{e.LPP}), i.e.
    \begin{equation*}
    		Q_{\ov z, N} = \sum_{i=1}^N \delta_{N^{1/\d}(x_i - \ov z)}. 
\end{equation*}
\begin{proposition}\label{p.PP convergence}
		Let $Q_{\overline{z},N}$ be the local point process given in (\ref{e.LPP}). Then, the correlation functions $R_k(y_1,\dots,y_k)$ for $(y_1,\dots,y_k)\subset (\R^\d)^k$ associated to $X_N$ via $x_i=\overline{z}+N^{-1/\d}y_i$, where $X_N$ is distributed according to $\P_{N,\beta}$, are given by 
		\begin{equation}\label{e.corr formula}
			R_k(y_1,\dots,y_k)=\frac{N!}{(N-k)!N^k}\rho_k(x_1,\dots,x_k)
		\end{equation}
		where $\rho_k$ denotes the $k$th marginal of $\P_{N,\beta}$
		\begin{equation}\label{e.marginal}
			\rho_k(x_1,\dots,x_k):=\frac{1}{\Kpart_{N,\beta}}\prod_{i=1}^k \mu_\theta(x_i)\int_{\R^{\d(N-k)}}e^{-\beta N^2 \Fenergy(X_N,\mu_\theta)}\prod_{i=k+1}^N \mu_\theta(x_i)~dx_i.
		\end{equation}
		Furthermore, if 
		\begin{enumerate}
			\item $\lim_{N \rightarrow \infty}R_k^N(y_1,\dots,y_k)=\mu_V(\ov z)^k$
			\item $\sup_{N \in \mathbb{N}}\sum_{k=1}^\infty \frac{1}{k!}\int_\Omega R_k^N(dy_1,\dots,dy_k)<+\infty$ for all compact $\Omega \subset \R^\d$
		\end{enumerate}
		then $Q_{\overline{z},N}$ converges weakly to a homogeneous Poisson point process of intensity $\mu_V(\ov z)$.
	\end{proposition}
    \begin{proof}
        Equation \eqref{e.corr formula} can be verified by a direct computation. 
        The second part of the claim follows from \cite[Lemma A.8]{Lam21}. 
    \end{proof}
    
	We turn now to a study of the $k$-point marginals of $\P_{N,\beta}$.
	
	\subsection{Proof of Theorem \ref{t.1}}
	We start with an explicit computation. 
	\begin{lemma} \label{l.kpoint}
		Let $Y_{k} := (y_{1}, ...y_{k}) \in \R^{\d \times k}$, and $x_{i} := \overline{z} + N^{-\frac{1}{\d}} y_{i}$. Then,
		the $k$-point marginals of $\P_{N,\beta}$ satisfy
		\begin{equation}\label{e.kpointmarg}
        \begin{split}
			&\frac{\rho_k(x_1,\dots,x_k)}{\prod_{i=1}^k \mu_{\theta}(x_i)} \\
            = &\frac{\Kpart_{N-k,\beta}}{\Kpart_{N,\beta}} e^{-\beta k^2 \Fenergy_k(X_k, \mu_\theta)}  \E_{N-k} \left[ e^{-\beta(N-k) \sum_{i=1}^k \h^{\fluct_{N-k}}(x_i) + \beta k(N-k) \int \h^{\mu_\theta}(x) ~\fluct_{N-k}(dx)} \right]
            \end{split}
		\end{equation}
		where 
		$$
		\fluct_{N-k} = \frac{1}{N-k} \sum_{i=k+1}^N \delta_{x_i} - \mu_{\theta},
		$$
		 the expectation $\E_{N-k}$ is over $(x_{k+1},\ldots,x_N)$ distributed via $\P_{N-k,\beta_N}$, and $\h$ is the Coulomb potential defined in (\ref{e.potential}). It follows that the $k$-point correlation functions are given by
		\begin{multline}\label{e.kpoint.Kpart}
			\frac{R_k(y_1,\dots,y_k)}{\prod_{i=1}^k \mu_{\theta}(x_i)}=\frac{N!}{(N-k)!N^k}\frac{\Kpart_{N-k,\beta}}{\Kpart_{N,\beta}} e^{-\beta k^2 \Fenergy_k(X_k, \mu_\theta)}\\
			\times \E_{N-k} \left[ e^{-\beta(N-k) \sum_{i=1}^k \h^{\fluct_{N-k}}(x_i) + \beta k(N-k) \int \h^{\mu_\theta}(x) ~\fluct_{N-k}(dx)} \right].
		\end{multline}
	\end{lemma}
	\begin{proof}
		A computation with the definition of $\Fenergy_N$ yields 
		\begin{align*}
			 & N^2 \Fenergy_N(X_k \cup X_{N-k}, \mu_\theta) -  (N-k)^2 \Fenergy_{N-k} (X_{N-k}, \mu_\theta)  \\ 
             = & k^2 \Fenergy_k(X_k,\mu_\theta) - k (N-k) \int \h^{\mu_\theta}(y) ~d\fluct_{N-k}(dy) + (N-k) \sum_{i=1}^k \h^{\fluct_{N-k}}(y_i).
		\end{align*}
		Inserting this into (\ref{e.marginal}) and simplifying yields (\ref{e.kpointmarg}), which coupled with (\ref{e.corr formula}) yields (\ref{e.kpoint.Kpart}).
	\end{proof}
	
	The main goal of the argument is to show that the quantities on the right hand side of \eref{kpointmarg} tend to one as $N \to \infty$ with $k$ fixed. We will handle the terms in the expectation in two steps. We first have a corollary of Proposition \ref{p.concentration} that is more tailored to computing the marginals. 
    For the remainder of the paper, we denote by $L^2(\P_{N,\beta})$ the space of real-valued functions on $\R^{\d \times N}$ that are square-integrable in the probability measure $\P_{N,\beta}$. For any $F \in L^2(\P_{N,\beta})$, we denote
    \begin{equation}
        \| F \|_{L^2(\P_{N,\beta})} := \sqrt{ \int |F|^{2} ~d \P_{N,\beta} }. 
    \end{equation}
	\begin{corollary} \label{c.exph}
		Assume \eref{growth} and consider $k \leq \frac{1}{10}N^{1/2}$ such that $\beta \leq \frac{1}{k C_0 N^{1/2} \log N}$ for a sufficiently large $C_0$. Then for any $F \in L^2(\P_{N,\beta})$,
		\begin{equation} \label{e.exph}
			\left|\E_{N,\beta} \left[ F(X_N) e^{-N\beta \sum_{i=1}^k \h^{\fluct}(y_i)} \right] - \E_{N,\beta} \left[ F(X_N) \right] \right| \leq 4k C_0 \beta N^{1/2} \log N \| F \|_{L^2(\P_{N,\beta})}.
		\end{equation}
	\end{corollary}

	\begin{proof}
  Given a non-negative integer $l$, let $\mcl G_\ell(\ep) = \{X_N : |\sum_{i=1}^k \h^{\fluct}(y_i)| \in [2^\ell k\ep,2^{\ell+1} k\ep) \}$. We will set $\ep = C_0 N^{-1/2}\log N$ for a fixed large $C_0 > 0$, to be determined later.

    We start by writing 
    \begin{equation}
    \label{eq:start}
        \begin{split}
             &\left|\E_{N,\beta} \left[ F(X_N) e^{-N\beta \sum_{i=1}^k \h^{\fluct}(y_i)} \right] - \E_{N,\beta} \left[ F(X_N) \right] \right| \\
            =& \left| \E_{N,\beta} \left[ F(X_N) \left(1 - e^{-N\beta \sum_{i=1}^k \h^{\fluct}(y_i)}\right) \right] \right| \\
            \leq & \E_{N,\beta} \left[ \1_{\cup_{\ell \geq 0} \mcl G_\ell(\ep)} |F(X_N)| \left| 1 - e^{-N\beta \sum_{i=1}^k \h^{\fluct}(y_i)} \right| \right] \\
            & \quad + \E_{N,\beta} \left[ |F(X_N)| \1_{\(\cup_{\ell \geq 0} \mcl G_\ell(\ep)\)^c} \left | 1 - e^{-N\beta \sum_{i=1}^k \h^{\fluct}(y_i)} \right | \right]
        \end{split}
    \end{equation}
    
    We being by estimating the third line in equation \eqref{eq:start}. For a fixed $l$ we have, via Cauchy-Schwarz
		\begin{align*}
			 &\E_{N,\beta} \left[ \1_{\mcl G_\ell(\ep)} |F(X_N)| \left| 1 - e^{-N\beta \sum_{i=1}^k \h^{\fluct}(y_i)} \right| \right] \\
             \leq &\| F \|_{L^2(\P_{N,\beta})} \(1+ e^{2^{\ell+1} kN\beta \ep}\) \sqrt{\P_{N,\beta}(\mcl G_\ell(\ep))} \\  \leq & 2 \sqrt{k} \| F \|_{L^2(\P_{N,\beta})} \(1+ e^{2^{\ell+1}k  N\beta \ep}\)e^{-\beta 2^{\ell-1} N^{3/2} \ep}
			\\  \leq & 2 \sqrt{k} \| F \|_{L^2(\P_{N,\beta})} \(1+ e^{C_0 2^{\ell+1}k  \beta N^{1/2}  \log N }\)e^{- C_0 \beta 2^{\ell-1} N \log N},
		\end{align*}
		where $C_0$ is chosen large enough that \pref{concentration} can be applied. Note that we used $\beta \leq \frac12 N^{-1/2}$ to simplify the error term from \pref{concentration}. Since $N^{1/2} > 10k$, we can bound the above by
		$$
		 \| F \|_{L^2(\P_{N,\beta})} e^{- \frac{C_0}{2} \beta 2^{\ell-1} N \log N},
		$$
		which decays rapidly with increasing $\ell$. We sum over $\ell = 0,1,\ldots$ to see
		\begin{equation} \label{e.onG.est}
        \begin{split}
		&\E_{N,\beta} \left[ \1_{\cup_{\ell \geq 0} \mcl G_\ell(\ep)} |F(X_N)| \left| 1 - e^{-N\beta \sum_{i=1}^k \h^{\fluct}(y_i)} \right| \right] \\
        \leq &\sum_{l} \E_{N,\beta} \left[ \1_{\mcl G_\ell(\ep)} |F(X_N)| \left| 1 - e^{-N\beta \sum_{i=1}^k \h^{\fluct}(y_i)} \right| \right]\\
        \leq &\| F \|_{L^2(\P_{N,\beta})}  e^{- \frac{C_0}{8} \beta N \log N}.
        \end{split}
		\end{equation}
		We now turn to estimating the fourth line in equation \eqref{eq:start}. Using again \pref{concentration} we have
		\begin{equation} \label{e.offG.est}
        \begin{split}
		\E_{N,\beta} \left[ |F(X_N)| \1_{\(\cup_{\ell \geq 0} \mcl G_\ell(\ep)\)^c} \left | 1 - e^{-N\beta \sum_{i=1}^k \h^{\fluct}(y_i)} \right | \right] &\leq \( e^{2k C_0 \beta N^{1/2} \log N}  - 1 \)\| F \|_{L^2(\P_{N,\beta})} \\
		&\leq  4k C_0 \beta N^{1/2} \log N \| F \|_{L^2(\P_{N,\beta})}.
        \end{split}
		\end{equation}
		For the last inequality, we used the assumed upper bound on $\beta$. Note that for $C_0$ sufficiently large we have
		$$
		 e^{- \frac{C_0}{8} \beta N \log N} \leq N^{-100} \leq 4k C_0 \beta N^{1/2} \log N
		$$
		since $\theta_\ast = \inf_N N \beta_N > 0$ by \eref{conf}, and so the error term from \eref{onG.est} can be absorbed into \eref{offG.est}.
	\end{proof}
	
	We will apply Corollary \ref{c.exph} except with the $\E_{N-k}$ expectation over $(x_{k+1},\ldots,x_N)$ as defined in \lref{kpoint} and with $F(X_{N-k})=e^{\beta k(N-k) \int \h^{\mu_\theta}(y) ~d\fluct_{N-k}(dy)} $. To do so, we need control on the second moment of $F$. We will use that the exponential moments of fluctuations of linear statistics for the Coulomb gas are well controlled by a priori energy estimates. In particular, we will use the following, which is \cite[Corollary 5.21]{S24}.
	
	\begin{proposition}(\cite[Corollary 5.21]{S24})
    \label{p.fluct}
		Suppose that $\varphi: \mathbb{R}^{d} \to \mathbb{R}$ is such that $\nabla \varphi \in L^2 \cap L^\infty$ and $\beta \leq 1/2$. Then,
		\begin{equation*}
			\left|\log \E \left(\exp\frac{\beta}{C\|\varphi\|^2} \Fluct[\varphi]^2\right)\right|\leq C\beta| \log \beta| N
		\end{equation*}
		where 
		\begin{equation*}
			\|\varphi\|=\|\nabla \varphi\|_{L^2}+\|\nabla \varphi\|_{L^\infty}.
		\end{equation*}
		and $\Fluct[\varphi]=N \int \varphi~d\fluct_N$.
	\end{proposition}
	We use this to obtain our desired bound on $F$ in $L^2$.
	\begin{proposition}\label{p.lower bound}
		Let $X_{N}$ be distributed according to $\P_{N,\beta}$ and $\beta \leq 1/2$. Then,
		\begin{equation}\label{e.UB}
			\E\left[e^{2\beta kN\int \h^{\mu_\theta(y)}\fluct(dy)}\right] \leq  \exp\left(Ck\beta(\log \beta)\sqrt{N}+Ck\right)
		\end{equation}
		and
		\begin{equation}\label{e.LB}
			\E\left[e^{2\beta kN\int \h^{\mu_\theta(y)}\fluct(dy)}\right] \geq \exp\left(-Ck\beta(\log \beta)\sqrt{N}-Ck\right).\end{equation}
		In particular, if $\beta \leq \frac{1}{\sqrt{N} \log N}$ then $\E\left[e^{2\beta kN\int \h^{\mu_\theta(y)}\fluct(dy)}\right]$ is uniformly bounded above and below as $N \to \infty$ by $e^{\pm Ck}$
	\end{proposition}
	\begin{proof}
		The main difficulty is that $\h^{\mu_\theta}$ is not in $H^1$ so we cannot use \pref{fluct} directly, and so we proceed with a cutoff argument.

	\textbf{Step 1:} Analysis of the regularized potential. 
    
    Let $\chi$ be a cutoff function that is identically one on a ball of radius $R$ and vanishing outside of a ball of radius $R+1$. First, notice that by Young's inequality we have
		\begin{equation*}
			-\lambda \beta \Fluct\left[k\h^{\mu_\theta}\chi\right]^2-\frac{\beta}{4\lambda} \leq \beta \Fluct\left[k\h^{\mu_\theta}\chi\right] \leq \lambda \beta \Fluct\left[k\h^{\mu_\theta}\chi\right]^2+\frac{\beta}{4\lambda},
		\end{equation*}
        for any $\lambda > 0$.
		We set $\lambda=\frac{1}{2C\left\|k\h^{\mu_\theta}\chi\right\|^2\sqrt{N}}$. Then, via H\"older's inequality and Proposition \ref{p.fluct}
		\begin{align*}
			\E \exp (2\beta \Fluct[k\h^{\mu_\theta}\chi]) &\leq e^{C\beta \sqrt{N}\left\|k\h^{\mu_\theta}\chi\right\|^2}\E \exp \left(\frac{\beta}{C\left\|k\h^{\mu_\theta}\chi\right\|^2\sqrt{N}}\Fluct[k\h^{\mu_\theta}\chi]^2\right) \\
			&\leq \exp \left(C\beta \sqrt{N}\left\|k\h^{\mu_\theta}\chi\right\|^2+C\beta (\log \beta)\sqrt{N}\right). 
		\end{align*}
		A computation shows that 
		\begin{equation*}
			\left\|kh^{\mu_\theta}\chi\right\|^2\lesssim k \log R
		\end{equation*}
		and hence 
		\begin{equation*}
			\E \exp (2\beta \Fluct[k\h^{\mu_\theta}\chi]) \leq \exp \left(C\beta (\log \beta)k\sqrt{N}\right)
		\end{equation*}
		by taking $R$ large enough, to be determined later. For the lower bound, using Jensen's inequality and Proposition \ref{p.fluct},
		\begin{align*}
			\E \exp \left(2\beta \Fluct\left[k\h^{\mu_\theta}\chi\right]\right) &\geq \left(\E \exp \left(\frac{\beta}{C\left\|k\h^{\mu_\theta}\chi\right\|^2} \Fluct\left[k\h^{\mu_\theta}\chi\right]^2\right)\right)^{-2\lambda C\left\|k\h^{\mu_\theta}\chi\right\|^2}e^{-\beta/4\lambda} \\
			&\geq \exp\left(-C\lambda \beta (\log \beta)N\left|\|k\h^{\mu_\theta}\chi\right\|^2-\frac{\beta}{4\lambda}\right).
		\end{align*}
		Hence
		\begin{equation*}
			\E \exp \left(2\beta \Fluct\left[k\h^{\mu_\theta}\chi\right]\right) \geq \exp\left(-Ck\lambda \beta (\log \beta)N\log^2 R-\frac{\beta}{4\lambda}\right).
		\end{equation*}
		and taking $\lambda=\frac{1}{\sqrt{N}}$ we conclude that 
		\begin{equation*}
			\E \exp \left(2\beta \Fluct\left[k\h^{\mu_\theta}\chi\right]\right) \geq \exp\left(-Ck \beta (\log \beta)\sqrt{N}\right).
		\end{equation*}
		
	\textbf{Step 2:} Conclusion.

        Let 
        \begin{equation}
            \zeta_{V} := h^{\mu_{V}} + V - c_{V},
        \end{equation}
        where $c_{V}$ is as in equation \eqref{eq:ELeq}. Note that $\zeta_{V}$ is non negative, and vanishes in the support of $\mu_{V}$. Since $h^{\mu_{\theta}}$ grows logarithmically at infinity (see \cite{AS22}), by equation \eqref{e.potgrowth} there is $R$ big enough such that $(1-\chi)|\h^{\mu_\theta}| \lesssim \zeta_V$. 

         Using the average localization bound from \cite[Corollary 5.26]{S24}
		\begin{equation*}
			\log \left|\E \left(\exp \frac{\beta N}{2}\sum_{i=1}^{N-k}\zeta_V(x_i)\right)\right|\leq CN
		\end{equation*}
		in the regime $\beta \rightarrow 0$, we have 
		\begin{equation*}
			\E \exp \left(\beta k \sum_i \h^{\mu_\theta}(1-\chi)(x_i) \right)  \leq C \E \exp \left(\beta k\sum_{i=1}^{N}\zeta_V(x_i)\right) \leq \exp(Ck).
		\end{equation*}

        On the other hand, via Jensen's inequality,
		\begin{equation*}
			\E \exp \left(\beta k\sum_i \h^{\mu_\theta}(1-\chi)(x_i)\right) \geq C \E \exp \left(-\beta k\sum_{i=1}^{N}\zeta_V(x_i)\right)\geq\exp(-Ck).
		\end{equation*}
		Combining all of the above yields (\ref{e.UB}) and (\ref{e.LB})
	\end{proof}

    We now have everything we need to complete the proof of Theorem \ref{t.1}. The main computation is the following asymptotic description of the correlation functions.
	\begin{proposition}\label{p. corrfunctions}
    Assume \eref{growth} and \eref{temp.zero}. Let $\overline{z} \in \R^2$ and let $R_k$ the $k$-point function associated to the local point process $Q_{\overline{z},N}$. For any $Y_k = (y_1,\ldots,y_k) \in (\R^2)^k$, let $x_i = \overline{z} + N^{-1/2} y_i$. Then we have
		\begin{equation}
			\frac{R_k(Y_k)}{e^{-\beta k^2 \Fenergy_k(X_k,\mu_\theta)} \prod_{i=1}^k \mu_\theta(x_i)}=1+O\left(\beta N^{1/2} \log N \right)
		\end{equation}
        as $N \to \infty$ with $k$ fixed, uniformly in $Y_k$.
		In particular,
		\begin{equation*}
			\frac{R_k(Y_k)}{\prod_{i=1}^k \mu_\theta(x_i)}\rightarrow 1,
		\end{equation*}
		and $Q_{\overline{z},N}$ converges to a homogeneous Poisson point process of intensity $\mu_V(\overline{z})$ whenever $\overline{z}$ is in the interior of $\supp \mu_V$.
	\end{proposition}
	\begin{proof}
    First, by equation \eref{kpointmarg},
\begin{equation*}
        \begin{split}
        &\frac{\rho_k(x_1,\dots,x_k)}{\prod_{i=1}^k \mu_{\theta}(x_i)} \\
            = &\frac{\Kpart_{N-k,\beta}}{\Kpart_{N,\beta}} e^{-\beta k^2 \Fenergy_k(X_k, \mu_\theta)}  \E_{N-k} \left[ e^{-\beta(N-k) \sum_{i=1}^k \h^{\fluct_{N-k}}(x_i) + \beta k(N-k) \int \h^{\mu_\theta}(x) ~\fluct_{N-k}(dx)} \right]
            \end{split}
		\end{equation*}
Applying \cref{exph} for $N-k$ with 
\begin{equation*}
F(X_{N-k})=e^{\beta k(N-k)\int \h^{\mu_\theta}(x)~\fluct_{N-k}(dx)}
\end{equation*}
we obtain
		\begin{multline*}
			\frac{\rho_k(X_k)}{\prod_{i=1}^k \mu_\theta(x_i)} = e^{-\beta k^2 \Fenergy_k(X_k,\mu_\theta)} \frac{\Kpart_{N-k,\beta}}{\Kpart_{N,\beta}}\times\\
			\( \E_{N-k} \left[ e^{\beta k (N-k) \int \h^{\mu_\theta}(y) ~\fluct_{N-k}(dy)} \right] + O(\beta N^{1/2} \log N)\|F\|_{L^2}\).
		\end{multline*}
        Utilizing Proposition \ref{p.lower bound} allows us to control $\|F\|_{L^2}$ uniformly in the regime $\beta \leq \frac{1}{\sqrt{N}\log N}$, which is guaranteed by the assumption (\ref{e.temp.zero}). 
        
        Next, we find by \lref{kpoint} and Tonelli's theorem that
		\begin{equation}
        \begin{split}
			&\frac{\Kpart_{N,\beta}}{\Kpart_{N-k,\beta}}\\
            =& \int \E_{N-k} \left[ e^{-\beta N^2 \Fenergy(X_k \cup X_{N-k}, \mu_\theta) + \beta (N-k)^2 \Fenergy (X_{N-k}, \mu_\theta)}\right] \mu_{\theta}^{\otimes k}(dX_k) \\
             = & \int e^{-\beta k^2 \Fenergy_k(X_k,\mu_\theta)} \E_{N-k} \left[ e^{-\beta(N-k) \sum_{i=1}^k \h^{\emp - \mu_\theta}(y_i) + \beta k(N-k) \int \h^{\mu_{\theta}}(y) \fluct_{N-k}(dy)} \right] \mu_\theta^{\otimes k}(dX_k).
             \end{split}
		\end{equation}
		We apply \cref{exph} and Proposition \ref{p.lower bound} once again and find
		\begin{equation} \label{e.postexph} 
        \begin{split}
        &\lefteqn{\frac{\rho_k(Y_k)}{\prod_{i=1}^k \mu_\theta(x_i)}}  \\ 
        = &\frac{e^{-\beta k^2 \Fenergy_k(X_k,\mu_\theta)} \( \E_{N-k} \left[ e^{\beta k (N-k) \int \h^{\mu_\theta}(y) \fluct_{N-k}(dy)} \right] + O(\beta N^{1/2} \log N) \)}{\int e^{-\beta k^2 \Fenergy_k(X_k',\mu_\theta)} \( \E_{N-k} \left[ e^{\beta k (N-k) \int \h^{\mu_\theta}(y)\fluct_{N-k}(dy)} \right] + O(\beta N^{1/2} \log N)  \) \mu_{\theta}^{\otimes k} (dX'_k)}.
        \end{split}
		\end{equation}
        Note that the error term in the denominator has an implicit constant independent of $X'_k$ by Proposition \ref{p.lower bound}. It is straightforward to see
        $$
        \int e^{-\beta k^2 \Fenergy_k(X_k',\mu_\theta)} \mu_{\theta}^{\otimes k} (dX'_k) = \exp\(O(\beta)\) = (1 + O(\beta))
        $$
        and so applying \pref{lower bound} we can see that \eref{postexph} is equal to
        $$
         (1+ O(\beta))e^{-\beta k^2 \Fenergy_k(X_k,\mu_\theta)} \frac{L + O(\beta N^{1/2} \log N)}{L + O(\beta N^{1/2} \log N)}
        $$
        for
        $$
        L = \E_{N-k} \left[ e^{\beta k (N-k) \int \h^{\mu_\theta}(y)\fluct_{N-k}(dy)} \right] = e^{O(1)},
        $$
        where the last bound follows from \pref{lower bound}. A simple Taylor approximation allows us to estimate
        \begin{align*}
        \frac{\rho_k(Y_k)}{\prod_{i=1}^k \mu_\theta(x_i)} &= e^{-\beta k^2 \Fenergy_k(X_k,\mu_\theta)} (1+ O(\beta)) (1 + O(\beta N^{1/2} \log N)) \\ &= e^{-\beta k^2 \Fenergy_k(X_k,\mu_\theta)} (1 + O(\beta N^{1/2} \log N)).
        \end{align*}
        Finally, we note that $\beta k^2 \Fenergy_k(X_k,\mu_\theta) = O(\beta \log N)$ for fixed $Y_k$ and that $\frac{(N-k)!N^k}{N!} = 1+O(1/N)$, so the same asymptotic holds for $R_k(Y_k)$ as $\rho_k(X_k)$. Since $\mu_\theta(y_i)\rightarrow \mu_V(\ov z)$ by (\ref{e. inttemp}) (see \cite[Theorem 1]{AS22}), we have the desired convergence of the correlation functions.
		
		Finally, the above computation coupled with the convergence $\mu_\theta(y_i)\rightarrow \mu_V(\ov z)$ (\cite[Theorem 1]{AS22}) tells us that there exists some $C>0$ such that 
		\begin{equation*}
			|R_k^N(y_1,\dots,y_k)|\leq (C\mu_V(\ov z))^k
		\end{equation*}
		for all $N$, and hence for any compact $\Omega \subset \R^2$
		\begin{equation*}
			\sup_{N \in \mathbb{N}}\sum_{k=1}^\infty \frac{1}{k!}\int_\Omega R_k^N(dy_1,\dots,dy_k) \leq \sum_{k=1}^\infty \frac{1}{k!}(C\mu_V(\ov z))^k|\Omega|<+\infty.
		\end{equation*}
		Convergence to a homogeneous Poisson point process of intensity $\mu_V(\ov z)$ then follows from Proposition \ref{p.PP convergence}.	\end{proof}

	\bibliographystyle{amsalpha}
	\bibliography{bibliography.bib}

\providecommand{\bysame}{\leavevmode\hbox to3em{\hrulefill}\thinspace}
\providecommand{\MR}{\relax\ifhmode\unskip\space\fi MR }
\providecommand{\MRhref}[2]{%
  \href{http://www.ams.org/mathscinet-getitem?mr=#1}{#2}
}
\providecommand{\href}[2]{#2}
\begin{thebibliography}{BBNY19}

\bibitem[AB19]{AB19}
Gernot Akemann and Sung-Soo Byun, \emph{The high temperature crossover for
  general 2d coulomb gases}, Journal of Statistical Physics \textbf{175}
  (2019), no.~6, 1043--1065.

\bibitem[ABG12]{ABG12}
Romain Allez, Jean-Philippe Bouchaud, and Alice Guionnet, \emph{Invariant beta
  ensembles and the gauss-wigner crossover}, Physical Review Letters
  \textbf{109} (2012), no.~9, 094102--.

\bibitem[AD14]{AD14}
Romain Allez and Laure Dumaz, \emph{From sine kernel to poisson statistics},
  Electronic Journal of Probability \textbf{19} (2014), 1--25.

\bibitem[AS74]{AS74}
Milton Abramowitz and Irene Stegun, \emph{Handbook of mathematical functions,
  with formulas, graphs, and mathematical tables,}, Dover Publications, Inc.,
  USA, 1974.

\bibitem[AS21]{AS21}
Scott Armstrong and Sylvia Serfaty, \emph{Local laws and rigidity for coulomb
  gases at any temperature}, Annals of Probability \textbf{49} (2021), no.~1,
  46--121.

\bibitem[AS22]{AS22}
\bysame, \emph{Thermal approximation of the equilibrium measure and obstacle
  problem}, Annales de la Facult{\'e} des sciences de Toulouse :
  Math{\'e}matiques, Serie 6 \textbf{31} (2022), no.~4, 1085--1110.

\bibitem[BBNY19]{BBNY19}
Roland Bauerschmidt, Paul Bourgade, Miika Nikula, and Horng-Tzer Yau, \emph{The
  two-dimensional coulomb plasma: quasi-free approximation and central limit
  theorem}, Advances in Theoretical and Mathematical Physics \textbf{23}
  (2019), no.~4, 841--1002.

\bibitem[BG13]{BG13}
Ga{\"e}tan Borot and Alice Guionnet, \emph{Asymptotic expansion of $\beta$
  matrix models in the one-cut regime}, Communications in Mathematical Physics
  \textbf{317} (2013), 447--483.

\bibitem[BG22]{BG16}
\bysame, \emph{Asymptotic expansion of $\beta$ matrix models in the multi-cut
  regime}, arXiv:1303.1045v7, October 2022.

\bibitem[BGP15]{BGP15}
Florent Benaych-Georges and Sandrine P{\'e}ch{\'e}, \emph{Poisson statistics
  for matrix ensembles at large temperature}, Journal of Statistical Physics
  \textbf{161} (2015), no.~3, 633--656.

\bibitem[BL18]{BL18}
Florent Bekerman and Asad Lodhia, \emph{Mesoscopic central limit theorem for
  general $\beta$-ensembles}, Ann. Inst. H. Poincare Probab. Statist.
  \textbf{54} (2018), no.~4, 1917--1938.

\bibitem[BLS18]{BLS18}
Florent Bekerman, Thomas Lebl{\'e}, and Sylvia Serfaty, \emph{{CLT for
  fluctuations of $\beta $-ensembles with general potential}}, Electronic
  Journal of Probability \textbf{23} (2018), 1--31.

\bibitem[BLZ23]{BLZ23}
Paul Bourgade, Patrick Lopatto, and Ofer Zeitouni, \emph{Optimal rigidity and
  maximum of the characteristic polynomial of wigner matrices},
  arXiv:2312.13335, December 2023.

\bibitem[BMP22]{BMP22}
Paul Bourgade, Krishnan Mody, and Michel Pain, \emph{Optimal local law and
  central limit theorem for $\beta$-ensembles}, Communications in Mathematical
  Physics \textbf{390} (2022), 1017--1079.

\bibitem[Bou23]{B21}
Jeanne Boursier, \emph{Optimal local laws and clt for the circular riesz gas},
  arXiv:2112.05881v3, February 2023.

\bibitem[Caf98]{C98}
L.~A. Caffarelli, \emph{The obstacle problem revisited}, Journal of Fourier
  Analysis and Applications \textbf{4} (1998), no.~4, 383--402.

\bibitem[CHM18]{CHM18}
Djalil Chafa{\"\i}, Adrien Hardy, and Myl{\`e}ne Ma{\"\i}da,
  \emph{Concentration for coulomb gases and coulomb transport inequalities},
  Journal of Functional Analysis \textbf{275} (2018), no.~16, 1447--1483.

\bibitem[CS07]{CS07}
Luis Caffarelli and Luis Silvestre, \emph{An extension problem related to the
  fractional laplacian}, Communications in Partial Differential Equations
  \textbf{32} (2007), no.~8, 1245--1260.

\bibitem[DRZ17]{DRZ17}
Jian Ding, Rishideep Roy, and Ofer Zeitouni, \emph{Convergence of the centered
  maximum of log-correlated gaussian fields}, The Annals of Probability
  \textbf{45} (2017), no.~6A, 3886--3928.

\bibitem[EHL21]{EHL21}
Matthias Erbar, Martin Huesmann, and Thomas Lebl{\'e}, \emph{The
  one-dimensional log-gas free energy has a unique minimizer}, Communications
  on Pure and Applied Mathematics \textbf{74} (2021), no.~3, 615--675.

\bibitem[FHK12]{FHK12}
Yan~V. Fyodorov, Ghaith~A. Hiary, and Jonathan~P. Keating, \emph{Freezing
  transition, characteristic polynomials of random matrices, and the riemann
  zeta function}, Physical Review Letters \textbf{108} (2012), no.~17,
  170601--.

\bibitem[Fro35]{F35}
Otto Frostman, \emph{Potentiel d'{\'e}quilibre et capacit{\'e} des ensembles
  avec quelques applications {\`a} la th{\'e}orie des fonctions}, Ph.D. thesis,
  Univ. Lund, 1935.

\bibitem[GZ19]{GZ19}
David Garc{\'\i}a-Zelada, \emph{{A large deviation principle for empirical
  measures on Polish spaces: Application to singular Gibbs measures on
  manifolds}}, Annales de l'Institut Henri Poincar{\'e}, Probabilit{\'e}s et
  Statistiques \textbf{55} (2019), no.~3, 1377 -- 1401.

\bibitem[GZPG24]{GZPG24}
David Garc{\'\i}a-Zelada and David Padilla-Garza, \emph{Generalized transport
  inequalities and concentration bounds for riesz-type gases}, Electronic
  Journal of Probability \textbf{29} (2024), 1--35.

\bibitem[Joh98]{J98}
Kurt Johansson, \emph{On fluctuations of eigenvalues of random hermitian
  matrices}, Duke Mathematical Journal \textbf{91} (1998), no.~1, 151--204.

\bibitem[KS09]{KS09}
Rowan Killip and Mihai Stoiciu, \emph{{Eigenvalue statistics for CMV matrices:
  From Poisson to clock via random matrix ensembles}}, Duke Mathematical
  Journal \textbf{146} (2009), no.~3, 361 -- 399.

\bibitem[Lam21a]{Lam21b}
Gaultier Lambert, \emph{Mesoscopic central limit theorem for the circular
  $\beta$-ensembles and applications}, Electronic Journal of Probability
  \textbf{26} (2021), 1--33.

\bibitem[Lam21b]{Lam21}
\bysame, \emph{Poisson statistics for gibbs measures at high temperature},
  Annales de l'Institut Henri Poincar{\'e}, Probabilit{\'e}s et Statistiques
  \textbf{57} (2021), no.~1, 326--350.

\bibitem[Leb16]{L16}
Thomas Lebl{\'e}, \emph{Logarithmic, coulomb and riesz energy of point
  processes}, Journal of Statistical Physics \textbf{162} (2016), no.~4,
  887--923.

\bibitem[Leb17]{L17}
Thomas Lebl{\'e}, \emph{Local microscopic behavior for $2$d coulomb gases},
  Probability Theory and Related Fields \textbf{169} (2017), 931--976.

\bibitem[LLZ24]{LLZ24}
Gaultier Lambert, Thomas Lebl{\'e}, and Ofer Zeitouni, \emph{Law of large
  numbers for the maximum of the two-dimensional coulomb gas potential},
  Electronic Journal of Probability \textbf{29} (2024), 1--36.

\bibitem[LS15]{LS15}
Thomas Lebl{\'e} and Sylvia Serfaty, \emph{Large deviation principle for
  empirical fields of log and riesz gases}, Inventiones mathematicae
  \textbf{210} (2015), no.~3, 645--757.

\bibitem[LS18]{LS18}
\bysame, \emph{Fluctuations of two dimensional coulomb gases}, Geometric and
  Functional Analysis \textbf{28} (2018), 443--508.

\bibitem[NT20]{NT20}
Fumihiko Nakano and Khanh~Duy Trinh, \emph{Poisson statistics for beta
  ensembles on the real line at high temperature}, Journal of Statistical
  Physics \textbf{179} (2020), no.~2, 632--649.

\bibitem[Pei24]{P24}
Luke Peilen, \emph{Local laws and a mesoscopic clt for $\beta$-ensembles},
  Communications on Pure and Applied Mathematics \textbf{77} (2024), no.~4,
  2452--2567.

\bibitem[Pei25]{peilen2025maximum}
\bysame, \emph{On the maximum of the potential of a general two-dimensional
  coulomb gas}, Electronic Communications in Probability \textbf{30} (2025),
  1--17.

\bibitem[PG23]{PG23}
David Padilla-Garza, \emph{Concentration inequality around the thermal
  equilibrium measure of coulomb gases}, Journal of Functional Analysis
  \textbf{284} (2023), no.~1, 109733.

\bibitem[PGPT24]{PGPT24}
David Padilla-Garza, Luke Peilen, and Eric Thoma, \emph{Emergence of a poisson
  process in weakly interacting particle systems}, arXiv:2405.02625, May 2024.

\bibitem[RRV11]{RRV11}
Jos{\'e}~A. Ram{\'\i}rez, Brian Rider, and B{\'a}lint Vir{\'a}g, \emph{Beta
  ensembles, stochastic airy spectrum, and a diffusion}, Journal of the
  American Mathematical Society \textbf{24} (2011), no.~4, 919--944.

\bibitem[RV07]{RV07}
Brian Rider and B{\'a}lint Vir{\'a}g, \emph{The noise in the circular law and
  the gaussian free field}, International Mathematics Research Notices
  \textbf{2007} (2007), rnm006.

\bibitem[Ser23]{S22}
Sylvia Serfaty, \emph{Gaussian fluctuations and free energy expansion for
  coulomb gases at any temperature}, Annales de l'Institut Henri Poincar{\'e},
  Probabilit{\'e}s et Statistiques \textbf{59} (2023), no.~2, 1074--1142.

\bibitem[Ser24]{S24}
\bysame, \emph{Lectures on coulomb and riesz gases}, arXiv:2407.21194, July
  2024.

\bibitem[SS12]{SS12}
Etienne Sandier and Sylvia Serfaty, \emph{From the ginzburg-landau model to
  vortex lattice problems}, Communications in Mathematical Physics \textbf{313}
  (2012), no.~3, 635--743.

\bibitem[SS15]{SS15-2}
Etienne Sandier and Sylvia Serfaty, \emph{2d coulomb gases and the renormalized
  energy}, The Annals of Probability \textbf{43} (2015), no.~4, 2026--2083.

\bibitem[ST97]{ST97}
Edward~B. Saff and Vilmos Totik, \emph{Logarithmic potentials with external
  fields}, Grundlehren der mathematischen Wissenschaften, Springer Berlin,
  Heidelberg, 1997.

\bibitem[SW71]{SW71}
Elias~M. Stein and Guido Weiss, \emph{Introduction to fourier analysis on
  euclidean spaces (pms-32)}, Princeton University Press, 1971.

\bibitem[Tho24]{T24}
Eric Thoma, \emph{Overcrowding and separation estimates for the coulomb gas},
  Communications on Pure and Applied Mathematics \textbf{77} (2024), no.~7,
  3227--3276.

\bibitem[Tho25]{T25}
Eric Thoma, \emph{A maximum principle for the coulomb gas: microscopic density
  bounds, confinement estimates, and high temperature limits},
  arXiv:2501.02733, January 2025.

\bibitem[VV09]{VV09}
Benedek Valk{\'o} and B{\'a}lint Vir{\'a}g, \emph{Continuum limits of random
  matrices and the brownian carousel}, Inventiones mathematicae \textbf{177}
  (2009), 463--508.

\end{thebibliography}

\end{document}